\documentclass[11pt]{amsart}
\usepackage{amsxtra,amssymb,amsmath,amsthm,amsbsy,amscd,amsfonts}
\usepackage{pb-diagram}
\usepackage{appendix}
\usepackage{enumerate}
\usepackage{latexsym}
\usepackage[Glenn]{fncychap}
\usepackage[all]{xy}
\usepackage[colorinlistoftodos]{todonotes} 

\usepackage{tikz}
\usetikzlibrary{matrix}

\theoremstyle{plain}

\numberwithin{equation}{section}

\newtheorem{theorem}{Theorem}[section]
\newtheorem{lemma}[theorem]{Lemma}
\newtheorem{definition-lemma}[theorem]{Definition-Lemma}
\newtheorem{proposition}[theorem]{Proposition}
\newtheorem{corollary}[theorem]{Corollary}
\newtheorem{definition}[theorem]{Definition}
\newtheorem{remark}[theorem]{Remark}

\theoremstyle{definition}
\newtheorem{example}[theorem]{Example}

\newcommand{\st}         {\;|\;}

\DeclareMathOperator{\image}{Im}
\DeclareMathOperator{\Ima}{Im}

\newcommand{\pr}         {\mathrm{pr}}
\newcommand{\Ann}        {\mathrm{Ann}}

\newcommand{\w}         {\omega}

\newcommand{\TM}            {T^*M}

\newcommand{\gpd}       {\mathcal{G}}

\newcommand{\leaf}         {\mathcal{O}}

\newcommand{\sour}        {\mathsf{s}}
\newcommand{\tar}         {{\mathsf{t}}}

\newcommand{\TP}          {T^*P}

\newcommand{\pp}         {\mathrm{pr}_2}

\newcommand{\ca}{[\![}
\newcommand{\cc}{]\!]}
\newcommand{\Cour}[1]      {[\![#1]\!]}

\newcommand{\Lie}        {\mathcal L}

\newcommand{\RR}      {{\mathbb R}}

\newcommand{\tve}{\tilde{\varepsilon}}
\newcommand{\ve}{\varepsilon}

\newcommand{\Osk}[1]{\Omega_{\mathrm{sk}}^{#1}(\mathcal{O},F_\mathcal{O})}
\newcommand{\Oskb}{\Omega_{\mathrm{sk}}^{\bullet}(\mathcal{O},F_\mathcal{O})}

\begin{document}

\title[]{On higher Dirac structures}

\author[]{Henrique Bursztyn, Nicolas Martinez Alba and Roberto Rubio}

\address{IMPA,
	Estrada Dona Castorina 110, Rio de Janeiro, 22460-320, Brasil }
\email{henrique@impa.br}

\address{IMPA,
	Estrada Dona Castorina 110, Rio de Janeiro, 22460-320, Brasil}
\curraddr{Universidad Nacional de Colombia - Sede Bogot\'a - Facultad de Ciencias - Departamento de Matem\'aticas - Carrera 30 No. 45-03, Bogot\'a D.C. - Colombia}
\email{nmartineza@unal.edu.co}

\address{IMPA,
	Estrada Dona Castorina 110, Rio de Janeiro, 22460-320, Brasil}
\curraddr{Weizmann Institute of Science, 234 Herzl St, Rehovot 7610001, Israel}
\email{roberto.rubio@weizmann.ac.il}

\begin{abstract}
We study higher-order analogues of Dirac structures, extending the multisymplectic structures that arise in field theory. We define higher Dirac structures as involutive subbundles of $TM+\wedge^k TM^*$ satisfying a weak version of the usual lagrangian condition (which agrees with it only when $k=1$).
Higher Dirac structures transversal to $TM$ recover the higher Poisson structures introduced in
\cite{BCI} as the infinitesimal counterparts of multisymplectic groupoids.
We describe the leafwise geometry underlying an involutive isotropic subbundle in terms of a  distinguished 1-cocycle in a natural differential complex, generalizing the presymplectic foliation of a Dirac structure.
We also identify the global objects integrating higher Dirac structures.
\end{abstract}

\maketitle

\section{Introduction}\label{sec:intro}

Dirac geometry \cite{Co} is an outgrowth of Poisson geometry, originally designed to describe the geometry of mechanical systems with constraints. Dirac structures provide a common framework
for the study of presymplectic and Poisson structures, and their recent applications include
generalized complex geometry, symmetries and moment maps, quantization, and more, see e.g. \cite{ABM,AM,BC,BCWZ,CGM,Gua}.


This paper concerns ``higher-order'' versions
of Poisson and Dirac structures, in the spirit of the higher-order symplectic forms that arise in classical field theory \cite{CID,CID2} and various other contexts, see e.g. \cite{BHR,MS,CR}.
Higher analogues of Dirac structures have been considered in field theory \cite{VYM}, in Nambu geometry \cite{Hag, BS}, as well as in the study of  $p$-branes in string theory \cite{zabzine}; a more systematic treatment is developed in \cite{Za}, which was one of the motivations for our work. Here we present another viewpoint to the subject, inspired by the theory of Lie groupoids: as discussed in \cite{BCI}, just as Poisson structures are infinitesimal versions of symplectic groupoids \cite{CDW} (analogously to how Lie algebras linearize Lie groups),
one is led to a natural notion of {\em higher Poisson structure} by considering the infinitesimal counterparts of {\em multisymplectic groupoids} (i.e., Lie groupoids equipped with compatible higher-order symplectic structures). In this paper, we take such higher Poisson structures as the starting point to develop a notion of higher Dirac structure.

\smallskip

\paragraph{\bf Description of the paper} To better explain our perspective, let us consider a manifold $M$, a positive integer $k$,
and the Whitney sum $TM + \wedge^{k} T^*M$, equipped with the fibrewise symmetric $\wedge^{k-1}T^*M$-valued pairing $$
\langle X + \alpha, Y+\beta \rangle := i_Y \alpha + i_X \beta,
$$
and the (higher) {\em Courant-Dorfman bracket} on the space of sections of $TM + \wedge^{k} T^*M$,
\begin{equation}\label{eq:CD}
  \Cour{X+\alpha,Y+\beta}:=[X,Y]+\mathcal{L}_X\beta - i_Y d\alpha.
\end{equation}
We use the notation $\pr_1: TM + \wedge^{k} T^*M \to TM$ and $\mathrm{\pr_2}: TM + \wedge^{k} T^*M\to \wedge^kT^*M$ for the natural projections.

For $k=1$, the pairing and bracket above make $TM+ T^*M$ into the {\em standard} Courant algebroid over $M$.
In this case, {\bf Dirac structures} are defined as {\em lagrangian} subbundles $L \subset TM+ T^*M$ (i.e., $L=L^\perp$ with respect to the pairing)
which are {\em integrable}, i.e., involutive with respect to the Courant-Dorfman bracket. Any closed 2-form $\omega\in \Omega^2(M)$ defines a Dirac structure given by the graph of the map $TM\to T^*M$, $X\mapsto i_X\omega$; in fact, closed 2-forms on $M$ are identified with Dirac structures $L$ satisfying $L\cap T^*M =\{0\}$, whereas Poisson structures on $M$ are the same as Dirac structures $L$ such that $L\cap TM=\{0\}$.

For $k> 1$, the very same definition leads to a possible notion of higher Dirac structure:
integrable, lagrangian subbundles $L \subset TM + \wedge^{k} T^*M$, as considered e.g. in \cite{Za}. For example, such subbundles satisfying $L\cap \wedge^k T^*M=\{0\}$ correspond to closed $(k+1)$-forms on $M$.
The condition $L\cap \wedge^k T^*M=\{0\}$, however, falls short of describing
higher Poisson structures: as observed in \cite{BCI}, these are not given by lagrangian subbundles. This led us to develop a new viewpoint to higher Dirac structures that weakens the lagrangian condition, in such a way that the resulting notion encompasses both closed higher-degree forms and higher Poisson structures, hence displaying a richer collection of examples.

Our starting observation in this paper is that there are natural ways to weaken the lagrangian condition for $k>1$, without changing it for $k=1$.
The distinct ways in which higher analogues of Dirac structures may be defined arise from equivalent ways to
describe ordinary Dirac structures; e.g.,  for a subbundle $L\subset TM+ T^*M$ it can be directly verified that the following are equivalent:
\begin{enumerate}
\item[(C1)] $L=L^\perp$;
\item[(C2)] $L\subseteq L^\perp$, and $L\cap TM = \pp(L)^\circ$.
\end{enumerate}
As it turns out, these two conditions are no longer equivalent for $k>1$, so (C2) gives a weaker way to extend the lagrangian condition (we discuss other possibilities in the Appendix).
We refer to a subbundle of $TM+ \wedge^{k}T^*M$ as in (C2) as {\em weakly lagrangian}, and we define {\bf higher Dirac structures} as weakly lagrangian subbundles which are integrable. Our study of higher Dirac structures
 relies on understanding two main ingredients: the pointwise linear algebra and the integrability of isotropic subbundles of $TM+ \wedge^{k}T^*M$. These lead to a description of the leafwise geometry of higher Dirac structures, extending the presymplectic foliation of usual Dirac structures,
 as well as their global integrations.


This paper is organized as follows.
We first consider higher Poisson structures \cite{BCI} in Section~\ref{sec:hpoiss}; we illustrate how they arise from multisymplectic structures in the presence of symmetries and provide a natural example in field theory, analogous to the Poisson brackets of classical mechanics.
We then introduce higher Dirac structures at the linear level,
in Section~\ref{sec:high-dirac-struct}. In Section~\ref{sec:involution}, we consider higher Dirac structures on manifolds, focusing on the additional integrability condition.
Any integrable isotropic subbundle $L\subset TM+\wedge^{k}T^*M$ has an underlying Lie algebroid, which gives rise to a (singular) foliation on $M$. Our main result (Theorem~\ref{theo:delta}) shows that each leaf $\leaf \hookrightarrow M$ carries a natural differential complex, denoted by $\Osk{\bullet}$, with a natural chain map to complex of differential forms on $\leaf$,
$$
\Osk{\bullet}\to \Omega^{\bullet + k}(\leaf),
$$
in such a way that $L$ is encoded by a distinguished 1-cocycle $\varepsilon_\leaf \in \Osk{1}$ on each leaf (Theorem~\ref{thm:integrability}). For ordinary Dirac structures, the previous chain map is an isomorphism, and this recovers their well-known presymplectic foliations.
We characterize various types of higher Dirac structures in $TM+ \wedge^k TM$, showing e.g. that those  projecting isomorphically onto $TM$ agree with closed $(k+1)$-forms on $M$, while higher Poisson structures are the same as higher Dirac structures intersecting $TM$ trivially.
In Section~\ref{sec:HPSgpd}, we relate higher Dirac structure to the theory of Lie groupoids by identifying their global counterparts (Theorem \ref{thm:integration of HD}), extending the integration of Dirac structures by presymplectic groupoids
of \cite{BCWZ}.

\smallskip

\paragraph{\bf Remark} The discussion in the paper extends with no extra cost to higher Dirac structures in $TM+ (\wedge^{k}T^*M\otimes \mathbb{R}^r)$, which incorporates poly-symplectic \cite{Gu} and poly-Poisson structures \cite{IMV,Ma}, as in \cite{Martz-th}. (More generally, one may consider $TM+ (\wedge^{k}T^*M\otimes E)$ for a vector bundle $E$ equipped with a (partial) flat connection, see \cite[Appendix]{Martz-th}.) For simplicity, we will restrict ourselves to the case $r=1$.


\tableofcontents



\paragraph{\bf Notation and conventions} For a vector space (or vector bundle) $V$, we denote the projections of
$V+ \wedge^k V^*$ to $V$ and $\wedge^kV^*$ by $\mathrm{pr}_1$ and $\pp$, respectively.
Given a subspace $A\subseteq \wedge^k V^*$, we use the notation
$$
A^\circ=\{ X \in V \st i_X\eta=0 \textrm{ for all } \eta \in A\}
$$
for its annihilator.
Analogously, for a subspace $E\subseteq V$, we have an annihilator $\Ann_{\wedge^k V^*}(E)$ in $\wedge^k V^*$;
whenever $k$ is clear from the context and there is no risk of confusion, we will use the simplified notation
$$
\Ann(E)=\{\alpha \in \wedge^k V^* \st i_Y\alpha=0 \textrm{ for all } Y\in E\}.
$$
In this paper, unless stated otherwise, {\em distributions} and {\em foliations} will be meant in the generalized sense of Stefan-Sussmann; we will refer to them as {\em regular} in case they have constant rank.


\section{Higher Poisson structures}\label{sec:hpoiss}

Before introducing higher Dirac structures, we recall the higher Poisson structures of \cite{BCI} (where more examples can be found).
Let $M$ be a smooth manifold and $k$ be a positive integer.

\begin{definition}\label{def:hpois}
  A {\em higher Poisson structure} (of order $k$) on $M$ is a subbundle $S\subseteq \wedge^k T^*M$ and a bundle map
  $\Lambda: S\to TM$ covering the identity map on $M$, such that
  \begin{itemize}
  \item[(a)] $S^\circ =\{0\}$,
  \item[(b)] $i_{\Lambda(\alpha)}\beta = - i_{\Lambda(\beta)}\alpha$, for all $\alpha,\beta \in S$,
  \item[(c)] the space $\Gamma(S)$ is involutive with respect to the bracket
    $$
    [\alpha,\beta]:= \Lie_{\Lambda(\alpha)}\beta - i_{\Lambda(\beta)}d \alpha = \Lie_{\Lambda(\alpha)}\beta - \Lie_{\Lambda(\beta)}\alpha - d(i_{\Lambda(\alpha)}\beta),
    $$
    and $\Lambda: \Gamma(S)\to \Gamma(TM)$ is bracket preserving.
  \end{itemize}

\end{definition}

When $k=1$, it follows \cite{BCI} from (a) that $S=T^*M$, while (b) implies that $\Lambda$ is a bivector field on $M$, and
(c) boils down to the integrability condition saying that $\Lambda$ is a Poisson structure.

A special class of examples is given by closed forms $\omega \in \Omega^{k+1}(M)$ that are non-degenerate,
in the sense that the map $\omega^\flat :TM \to \wedge^k T^*M$, $X\mapsto i_X\omega$, is injective. Such forms are referred to as {\em multisymplectic} \cite{CID} of degree $k+1$, or simply {\em $k$-plectic} \cite{CR}.
One can regard them as higher Poisson structures with $S=\Ima(\omega^\flat)$ and $\Lambda=(\omega^\flat)^{-1}$.
In fact, $k$-plectic forms are the same as higher Poisson structures $(S,\Lambda)$ of order $k$ for which the map $\Lambda: S\to TM$ is an isomorphism.

As discussed in \cite{BCI}, higher Poisson structures arise as the infinitesimal counterparts
of {\em multisymplectic groupoids}, i.e., Lie groupoids equipped with a multiplicative multisymplectic form, a fact that naturally extends the well-known connection between symplectic groupoids and ordinary Poisson structures \cite{CaFe,CDW,CF}; we will return to Lie groupoids in Section~\ref{sec:HPSgpd}.

We now show how higher Poisson structures may arise as quotients of multisymplectic structures by symmetries,
which leads to a concrete example from classical field theory \cite{Martz-th}.

\begin{example}\label{ex:quot} Let $M$ be an $n$-dimensional manifold equipped with a $k$-plectic form $\omega$. Consider an action $\psi$ of a Lie group $G$ on $M$ satisfying $\psi_g^*\omega = \omega$ for all $g\in G$. Suppose that the action is free and proper, so that $q: M\to M/G$ is a principal bundle, and let $\mathcal{V}\subseteq TM$ be the vertical bundle (tangent to the $G$-orbits, so $\dim(G)=\textup{rk}(\mathcal{V})$). Define $S=\Ima(\omega^\flat)\subseteq \wedge^k T^*M$, which satisfies $S^\circ$= \{0\} since $\omega$ is nondegenerate. Let us assume that the following two conditions hold:
  \begin{itemize}
  \item[(1)] $S\cap \wedge^k \Ann(\mathcal{V})$ has constant rank,
  \item[(2)] $(S\cap \wedge^k \Ann(\mathcal{V}))^\circ \subseteq \mathcal{V}$, which turns out to be equivalent to  requiring that  $\dim(G)\leq n-k$
or $\dim(G)=n$ (see \eqref{eq:doubleAnn-standard} below for details, noticing that $(S\cap \wedge^k \Ann(\mathcal{V}))^\circ = S^\circ + (\wedge^k \Ann(\mathcal{V}))^\circ = (\wedge^k \Ann(\mathcal{V}))^\circ$).
  \end{itemize}
The bundle map $Tq: TM \to q^*T(M/G)$, whose kernel is $\mathcal{V}$, is such that its transpose defines an isomorphism $(Tq)^*: q^* T^*(M/G) \to \Ann(\mathcal{V})$, which extends to
  $$
  q^* \wedge^k T^*(M/G) \stackrel{\sim}{\to} \wedge^k \Ann(\mathcal{V}).
  $$
  By condition (1) and the invariance of $\omega$, we see that there is a well-defined subbundle $S_{red}\subseteq \wedge^k T^*(M/G)$ that corresponds to $S\cap \wedge^k \Ann(\mathcal{V})$ under the previous isomorphism. Moreover, $S_{red}^\circ =\{0\}$ follows from condition (2). Finally, the $G$-invariance of $\omega$ implies that $(\omega^\flat)^{-1}$ induces a map $\Lambda_{red}: S_{red}\to T(M/G)$ so that $(S_{red},\Lambda_{red})$ is a higher Poisson structure on $M/G$ of order $k$.
\end{example}

\medskip

\paragraph{\bf A higher Poisson structure from classical field theory}

In classical field theory \cite{GIMM,KT}, one has the following geometric set-up: a configuration bundle given by a fibre bundle $P\overset{\pi}{\to} B$, and a volume form $\eta$ on the $m$-dimensional manifold $B$.  We let $VP= \ker(T\pi)\subset TP$ be the vertical bundle.

The phase space of this theory is the affine dual of the first jet bundle of $P$, which can be identified with
$$
M = \wedge^m_2\TP =\{ \beta \in \wedge^m \TP\,|\, i_vi_u\beta=0 \,\textrm{ for all }\, u,v \in VP\},
$$
called the {\em extended phase bundle}. We let $q_1: M\to P$ be the natural projection.
The manifold $M$ carries a natural $m$-plectic form  $\omega=-d\theta$, where $\theta$ is the `tautological' $m$-form on $M$ given by
$$
\theta_\beta (X_1,\ldots,X_m) = \beta(Tq_1( X_1),\ldots,Tq_1 (X_m)),
$$
see e.g. \cite{K}. Additionally, $M$ carries an $\mathbb{R}$-action,
$$
\beta \stackrel{r}{\mapsto} \beta + r \pi^*\eta,\;\;\; r\in \mathbb{R},
$$
whose quotient $Z = M/\mathbb{R}$ is a manifold so that the projection $q:M \to M/\mathbb{R}$ is a surjective submersion. This action preserves $\omega$.  The manifold $Z$ is a vector bundle over $P$, called the {\em reduced bundle}, and inherits a higher Poisson structure, as in Example~\ref{ex:quot} above. We will express it explicitly in coordinates. To simplify the notation, we may avoid the use of $\wedge$ in the local description of forms in this example.

Consider coordinates $(x_1,\ldots,x_m)$ on $B$ such that $\eta=dx_1\dots dx_m$, and adapted coordinates $(x_1,\dots,x_m,y_1,\dots,y_n)$ on $P$. Any $\beta\in M$ can be locally written as
$$
\beta = p  \pi^*\eta  + \sum_{l,k} p_{lk} dy_l \wedge \pi^*\eta_k,
$$
where $\eta_k = dx_1 \ldots  dx_{k-1} dx_{k+1} \ldots  d x_m$,
so $M$ has local coordinates $(x_i,y_j,p,p_{lk})$. The canonical $m$-plectic form on $M$ is given by
$$
\omega = -dp  dx_1  \ldots   dx_m - \sum_{l,k} dp_{lk}  dy_l  dx_1 \ldots   dx_{k-1}  dx_{k+1}  \ldots   d x_m.
$$

The $\mathbb{R}$-action on $M$ is such that $\mathcal{V} =  \textup{span}\{ \frac{\partial}{\partial p} \}$,
so $\Ann(\mathcal{V}) = \textup{span}\{dx_i,dy_j,dp_{lk}\}$. In coordinates, the quotient $Z=M/\mathbb{R}$ is given by
$(x_i,y_j,p_{lk})$, and the quotient map is simply $q(x_i,y_j,p,p_{lk})=(x_i,y_j,p_{lk})$. The induced higher Poisson structure on $Z$ is defined by the subbundle
\begin{equation}\label{eq:Sred}
  S_{red} = \textup{span}\{\alpha_l, dx_1\ldots d x_m, \gamma_{li}\} \subset \wedge^mT^*Z,
\end{equation}
where $\alpha_l = \sum_k dp_{lk}dx_1\ldots dx_{k-1} dx_{k+1} \ldots d x_m$, and $\gamma_{li}=dy_ldx_1\ldots dx_{i-1} dx_{i+1} \ldots d x_m$, for $l=1,\ldots,n$ and $i=1,\ldots,m$, and the bundle map $\Lambda_{red}: S_{red}\to TZ$,
\begin{equation}\label{eq:Lamred}
  \Lambda_{red}(\alpha_l) = \frac{\partial}{\partial y_l}, \;\;\;\; \Lambda_{red}(dx_1\ldots d x_m) = 0, \;\;\;\; \Lambda_{red}(\gamma_{li}) = -\frac{\partial}{\partial p_{li}}.
\end{equation}

\begin{example}
  The simplest case of the above construction is that of `time-dependent classical mechanics': for a configuration manifold $Q$, we set $P=\mathbb{R}\times Q \to \mathbb{R}$ the trivial bundle over $B=\mathbb{R}$, endowed with volume form $\eta = dx$. Then $M=  T^*\mathbb{R} \times T^*Q$, with canonical symplectic form $\omega = dx\wedge dp +  \sum_j dy_j\wedge dp_j$, and $Z= T^*Q \times \mathbb{R}$. The higher Poisson structure \eqref{eq:Sred} and \eqref{eq:Lamred} on $Z$, in this case, is just the ordinary Poisson structure
  $$
  \Lambda = \sum_j \frac{\partial}{\partial p_j}\wedge \frac{\partial}{\partial y_j}.
  $$
\end{example}

Our next goal is seeing how higher Poisson structures $(S,\Lambda)$ naturally lead to higher Dirac
structures, given by their graphs
$$
L=\{(\Lambda(\alpha),\alpha)\;|\; \alpha\in S \}\subset TM+ \wedge^k T^*M.
$$
We will first consider higher Dirac structures at the level of linear algebra.

\section{Linear theory of higher Dirac structures}\label{sec:high-dirac-struct}

Let $V$ be an $n$-dimensional vector space, take an integer $k$ with $1\leq k\leq  n-1$  (cf. Remark~\ref{rem:k=0,n}), and consider the space
$V+\wedge^k V^*$  equipped with the $\wedge^{k-1}V^*$-valued pairing
\begin{equation}\label{eq:pairing}
  \langle X+\alpha,Y+\beta \rangle= i_X\beta +i_Y\alpha.
\end{equation}

We start by giving a characterization of isotropic subspaces of $V+\wedge^k V^*$ that extends
the well-known fact (see, e.g., \cite{Co}) that lagrangian subspaces of $V+ V^*$ are determined by pairs
$(E,\varepsilon)$, where $E\subseteq V$ is a subspace and $\varepsilon\in \wedge^2 E^*$, via
$$
L(E,\varepsilon)= \{ X+\alpha \st X\in E \;\textrm{ and }\; \alpha|_{E}=i_X\varepsilon \}.
$$

\subsection{Isotropic subspaces}

Let $L\subset V+\wedge^k V^*$ be an isotropic subspace with respect to \eqref{eq:pairing}, i.e., $L\subseteq L^\perp$. We use the notation
$$
E=\pr_1(L)\subseteq V,\;\;\; \; A_L = L\cap \wedge^k V^*.
$$
Note that $A_L\subseteq \Ann(E)=L^\perp\cap \wedge^k V^*$ and $\pr_2(L)^\circ = L^\perp\cap V$.

For each  $X\in E$, take $X+\alpha\in L$, the element $X+\alpha'$ belongs to $L$ if and only if $\alpha-\alpha'\in A_L \subseteq \Ann(E)$. In order to describe $L$, we consider the map
\begin{equation}\label{eq:def-epsilon}
  \varepsilon : E \to \wedge^k V^*/A_L,
\end{equation}
where $\varepsilon(X)=\alpha+A_L$ if and only if $X+\alpha\in L$, which gives the possible $k$-forms going together with an element of $E$. By composing $\varepsilon$ with the map
\begin{equation}\label{eq:iotaE}
  \iota_E:\wedge^k V^*/A_L \to E^*\otimes \wedge^{k-1}V^*,\;\;\; \alpha+A_L \mapsto (X \mapsto i_X\alpha),
\end{equation}
we see that the fact that $L$ is isotropic means that $\iota_E \circ \varepsilon$ is skew symmetric, i.e., defines an element in $\wedge^2 E^* \otimes \wedge^{k-1}V^*$.

\begin{definition}\label{def:E-skew}
  We say that a map $\varepsilon : E \to \wedge^k V^*/A_L$ is {\em $E$-skew} when $\iota_E\circ \varepsilon$ belongs to  $\wedge^2 E^* \otimes \wedge^{k-1}V^*$.
\end{definition}

We summarize the discussion in the following proposition.

\begin{proposition}\label{prop:isotropic subsp}
  Given subspaces  $E\subseteq V$, $A_L\subseteq \Ann(E)$ and a $E$-skew map $\varepsilon : E \to \wedge^k V^*/A_L$, the subspace
  \[
  L(E,A_L,\varepsilon)=\{ X + \alpha \st X\in E, \alpha + A_L \in\varepsilon(X) \} \subset V+ \wedge^k V^*
  \]
  is isotropic. Conversely, any isotropic subspace $L$ is of the form $L(E,A_L,\varepsilon)$ with $E=\pr_1(L)$, $A_L=L\cap \wedge^k V^*$ and $\varepsilon$ defined as in \eqref{eq:def-epsilon}.
\end{proposition}

Let $L=L(E,A_L,\varepsilon)$ be isotropic.
Observe that if $\dim(E)> n-k$, then $\Ann(E)=\{0\}$, and hence $A_L=\{0\}$. It follows that
$L$ must be necessarily the graph of an $E$-skew map $\varepsilon: E\to \wedge^k V^*$,
\begin{equation}\label{eq:LE}
  L= L(E,0,\varepsilon) = \{ X + \varepsilon (X) \,|\, X\in E \}.
\end{equation}
Note also that, if $n>\dim(E)> n-k$, then $\Ann(E)^\circ = V \supsetneq E$.
On the other hand, if $\dim (E) \leq n-k$, or $\dim(E) = n$, then
\begin{equation}\label{eq:doubleAnn-standard}
  \Ann(E)^\circ = E.
\end{equation}
We distinguish isotropic subspaces according to these properties:

\begin{definition}\label{def:standard}
  An isotropic subspace $L(E,A_L,\varepsilon)$ is said to be {\em standard} when
  $\Ann(E)^\circ = E$, which happens if and only if
  \begin{equation}\label{eq:dim constraint}
    \dim(E)=n \text{\ \ \ or\ \ \ }\dim(E)\leq n-k.
  \end{equation}
  Otherwise we call it non-standard.
\end{definition}

As mentioned above, when $L$ is non-standard, it must be of the form $L(E,0,\varepsilon)$ with $E \neq V$.

\begin{remark}\label{rem:dim}
  For later use, we note that, given an isotropic $L\subset V+ \wedge^k V^*$, we have a constraint on the dimension of $\pp(L)$: if $E\neq 0$ then
  \begin{equation}\label{eq:dim-constraint}
    \dim (\pp(L))        \leq       \dim (E) + { n-1 \choose k }.        
  \end{equation}
  To check that, take $X+\alpha\in L$ with $X\neq 0$. Then $i_X\beta = -i_Y\alpha$ for any $Y+\beta\in L$. So the
  images of the maps $i_X |_{\pp(L)}: \pp(L) \to \wedge^{k-1}V^*$ and $\alpha|_E: E \to \wedge^{k-1}V^*$
  must be the same. It follows that
  \begin{align}\label{eq:dimA}
    \dim (\pp(L)) + \dim (\ker(\alpha|_E)) & = \dim(E) + \dim (\ker(i_X|_{\pp(L)}))\\ & \leq \dim(E) + \dim(\ker (i_X)).\nonumber
  \end{align}
  As $\dim (\ker(i_X: \wedge^kV^* \to \wedge^{k-1}V^*))={ n-1 \choose k }$ for $X\neq 0$, the constraint \eqref{eq:dim-constraint} follows.
\end{remark}

\subsection{Weakly lagrangian subspaces}\label{sec:weak-lagr-subsp}

 In the theory of higher Dirac structures, particular types of isotropic subspaces $L\subset V + \wedge^k V^*$  play a key role. We say that $L$ is {\em lagrangian} if $L=L^\perp$. Clearly $L=V$ and $L=\wedge^k V^*$ are lagrangian subspaces. More examples may be obtained as follows.

\begin{example}\label{ex:form}
  Let $\omega \in \wedge^{k+1}V^*$. Then its graph,
  $$
  L= \mathrm{gr}(\omega) = \{X + i_X\omega \,|\, X\in V\} \subset V+ \wedge^k V^*,
  $$
  is lagrangian.
\end{example}
Note that, while for any isotropic subspace $L\subset V + \wedge^k V^*$ we have $L\cap \wedge^k V^*\subseteq \Ann(E)$, for lagrangian subspaces we have $L\cap \wedge^k V^* = L^\perp \cap \wedge^k V^* = \Ann(E)$, i.e.,
\begin{equation}\label{eq:lag}
  A_L=\Ann(E).
\end{equation}

As mentioned in the introduction, the lagrangian condition is too strong to encompass higher Poisson structures.
So we introduce the following

\begin{definition}\label{def:weakly-lagrangian}
  An isotropic subspace $L\subset V+\wedge^k V^*$ is {\em weakly lagrangian} if
  \begin{equation}\label{eq:wlag}
    L\cap V= \pp(L)^\circ =  L^\perp \cap V.
  \end{equation}
  (Note that the second equality always holds.)

\end{definition}

Clearly any lagrangian space is weakly lagrangian.
For $k=1$, one may directly check that these two notions are equivalent.
Just as lagrangian subspaces of $V+ V^*$ define Dirac structures on a vector space $V$, weakly lagrangian subspaces of  $V+\wedge^k V^*$ will be {\em higher Dirac structures} (of order $k$)
on $V$.

\begin{remark}
  \label{rem:k=0,n}
  We consider $V+\wedge^{k} V^*$ for $1\leq k\leq \dim(V)-1$ as the cases $k=0,\dim(V)$ can be easily described on their own. For $k=0$, we have $V+\mathbb{R}$ and the pairing vanishes. Hence, any subbundle is isotropic and the only weakly lagrangian ones are $V$ and the total. For $k=n$, we have $V + \det V^*$. In this case, the isotropic subspaces are the subspaces of $V$ and $\det V^*$, whereas the only weakly lagrangian ones are $V$ and $\det V^*$. \end{remark}

Being isotropic, any weakly lagrangian subspace is of the form $L(E,A_L,\varepsilon)$. Since $L\cap V = \ker(\varepsilon) \subseteq \pp(L)^\circ$ always holds, we see that the weakly lagrangian condition in the previous definition is equivalent to
\begin{equation}\label{eq:weak-Lagrangian sbspace}
  \pp(L)^\circ \subseteq \ker (\varepsilon),
\end{equation}
or, alternatively,
$$
(\pr^{-1} (\Ima(\varepsilon)))^\circ \subseteq \ker(\varepsilon),
$$
where $\pr: \wedge^k V^* \to \wedge^k V^*/A_L$ is the natural projection.

\begin{example}\label{ex:lHP}
  For a pair $(S, \Lambda)$, where $S\subseteq \wedge^k V^*$ is a subspace and $\Lambda: S\to V$ is a linear map, consider its `graph',
  $$
  L = \{ \Lambda(\alpha) + \alpha \,|\, \alpha\in S\} \subset V+ \wedge^k V^*.
  $$
  The condition that $L$ is isotropic amounts to the property
  $$
  i_{\Lambda(\alpha)}\beta = -i_{\Lambda(\beta)}\alpha,
  $$
  and $L$ is weakly lagrangian if and only if $S^\circ = L^\perp \cap V = L\cap V = \{0\}$.
  We refer to a pair $(S,\Lambda)$ with these properties as a {\em higher Poisson structure} on $V$ (cf. (a) and (b) in Definition~\ref{def:hpois}).
  In other words, linear-algebraic versions of higher Poisson structures are examples of weakly lagrangian subspaces.
  As mentioned in Section~\ref{sec:hpoiss}, particular examples are given by elements $\omega\in \wedge^{k+1}V^*$ which are non-degenerate: in this case, $S=\Ima(\omega^\flat)$,  $\Lambda=(\omega^\flat)^{-1}$, and the corresponding $L$ is, moreover, lagrangian. On the other hand, whenever $S=\wedge^k V^*$, the subspace $L$ is also automatically lagrangian: for $Y+\beta \in L^\perp$, we have $i_Y\alpha = - i_{\Lambda(\alpha)}\beta = i_{\Lambda(\beta)}\alpha$ for all $\alpha \in \wedge^kV^*$, so $Y=\Lambda(\beta)$.
  A full characterization of higher Poisson structures whose graphs are lagrangian will be given in Proposition~\ref{prop:HP} below.
\end{example}

\begin{remark} We point out that higher Poisson structures $(S,\Lambda)$ with $S=\wedge^k V^*$ and $\Lambda\neq 0$ are very restricted: the only possibilities are $k=1$ and $k=n-1$, which occur when $\Lambda$ is defined by a bivector field or top-degree multivector, respectively. This is proven in \cite[Prop.~3.4]{Za}, and we give a short alternative argument. From \eqref{eq:dim-constraint} we have
  $$
  {n \choose k} \leq n +  { n-1 \choose k},
  $$
  which restricts the possible values of $k$ to $0,1,2,n-1,n$. The cases $k=0,n$  are ruled out by Remark~\ref{rem:k=0,n}.
  For $k=2$, we may assume that $n>3$. As $\Lambda\neq 0$, there exists $X+\alpha\in L$ with $X\neq 0$,
  for which \eqref{eq:dimA} becomes $n-1 + \dim (\ker(\alpha|_E)) = \dim(E)$.
  In particular, $n-1\leq \dim (E)$, so $\Ann(E)=\{0\}$, as $k=2$, and
  consequently $A_L=\{0\}$. But this is a contradiction with the map
  $\varepsilon:E\to \wedge^2 V^*/A_L=\wedge^2 V^*$ being surjective, as
  ${n\choose 2}>n$ for $n>3$. So $k=2$ is not possible.
\end{remark}

\begin{example}\label{ex:lHP1}
For any subspace $S\subseteq \wedge^k V^*$, $L =S^\circ + S$ is weakly lagrangian  with $\varepsilon$ identically  $0\in\wedge^k T^*M/S$.  In particular, any subspace $S\subseteq \wedge^k V^*$ such that $S^\circ=\{0\}$ is weakly lagrangian (this is also a special case of the previous example, with $\Lambda=0$). A concrete example is the span of a non-degenerate $k$-form $\alpha \in \wedge^k V^*$, i.e., $S=\langle \alpha \rangle = L$, which is a dimension-one subspace. In the particular case $L = S\subseteq \wedge^k V^*$, we have that $L^\perp = \wedge^k V^*$, so $L$ is not lagrangian unless $S=\wedge^k V^*$.
\end{example}

Hence, weakly lagrangian subspaces of $V+ \wedge^k V^*$ are not necessarily lagrangian.

While Dirac structures on a vector space can be always restricted to subspaces \cite[Sec.~1.4]{Co}, the situation
for higher Dirac structures is more subtle. Let $L=L(E,A_L,\varepsilon)$ be an isotropic subspace of $V+\wedge^k V^*$ and consider a subspace $i:W\hookrightarrow V$. Denote by $i^*\varepsilon$ the composition
$$
E\cap W\stackrel{i|_{E\cap W}}{\longrightarrow} E \stackrel{\varepsilon}{\longrightarrow} \wedge^kT^*M/A_L \stackrel{i^*}{\longrightarrow} \wedge^k W^*/i^*A_L.
$$
Then
$$
L_W=L(E\cap W,i^* A_L,i^*\varepsilon)
$$
is an isotropic subspace of $W+\wedge^k W^*$. However, even if $L$ is weakly lagrangian, the subspace $L_W$ need not be weakly lagrangian,  as the next example shows.

\begin{example}
Consider $V=\RR^5$ with basis $\{e_1,\ldots,e_5\}$ and dual basis $\{e^1,\ldots,e^5\}$.  Let $L=L(E,0,\varepsilon)$ with $E=\textup{span}\{ e_1,e_2 \}$ and  $\varepsilon=e^1\otimes( e^4\wedge e^5) + e^2\otimes (e^3\wedge e^4)$, which is higher Poisson and not lagrangian. The restriction to $W_1=\textup{span}\{ e_1,e_2 \}$ gives $L_{W_1}=W_1$, which is lagrangian. The restriction  to $W_2=\textup{span}\{ e_2,e_3,e_4 \}$ gives $L_{W_2}=L(E_2,0,\varepsilon_2)$ with $E_2=\textup{span}\{e_2 \} $ and $\varepsilon_2=e^2\otimes (e^3 \wedge e^4)$, which is higher Poisson.  The restriction  to $W_3=\textup{span}\{e_1,e_2,e_3,e_4 \}$ gives $L_{W_3}=L(E_3,0,\varepsilon_3)$ with $E_3=\textup{span}\{e_1,e_2 \} $ and $\varepsilon_3=e^2\otimes (e^3 \wedge e^4)$, which is weakly lagrangian but not higher Poisson, as $L_{W_3}\cap W_3 = \textup{span}\{e_1\}$. Finally, the restriction to $W_4=\textup{span}\{e_1, e_2,e_4\}$ gives $L_{W_4}=\textup{span}\{e_1, e_2\}$, which is not weakly lagrangian, as $\pr_2(L)^\circ=\{0\}^\circ = W_4\neq \textup{span}\{e_1, e_2\}$.
\end{example}

Dirac structures $L\subset V+V^*$ such that $L\cap V^*=\{0\}$ are graphs of elements in $\wedge^2V^*$.
For weakly lagrangian subspaces in $V+\wedge^k V^*$, we have the following:

\begin{proposition}\label{prop:inters1}
  A weakly lagrangian subspace $L\subset V+ \wedge^k V^*$ such that $A_L = \{0\}$, i.e., $L\cap \wedge^k V^*= \{0\}$, 
  is of the form $L(E,0,\varepsilon)$ (as in \eqref{eq:LE}) with $\Ima(\varepsilon)^\circ = \ker (\varepsilon)$, and it is lagrangian if and only if $E=V$.
\end{proposition}

So in the lagrangian case the condition $L\cap \wedge^k V^*= \{0\}$ implies that $L$ is the graph of a $(k+1)$-form, as in Example~\ref{ex:form}.

\begin{proof}
  The first assertion is an immediate consequence of Proposition~\ref{prop:isotropic subsp}.

  For the second assertion, it is enough to prove that $\pr_1(L^\perp)=V$.
  So, for any $X\in V$, we must find $\alpha\in \wedge^k V^*$ such that
  $X+\alpha\in L^\perp$. By orthogonality with $Y+\varepsilon(Y)\in L$ we
  have, for any $Y\in E$,
  $$
  i_Y\alpha=-i_X \varepsilon(Y).
  $$  
  This prescribes what $\iota_E (\alpha):E\to \wedge^{k-1} V^*$ should be. It remains to
  show that this can be extended to a $k+1$-form $\alpha$. Choose a complement
  $C$ to $E$ in $V$. For $Y\in C$, $i_Y \alpha\in \wedge^{k-1} V^*$ evaluated on $X_1,
  \ldots, X_{k-1}\in V$ with $X_j\in E$ is given by
  $$
  i_Y\alpha (X_1,\ldots, X_j,\ldots , X_{k-1}) = -
  i_{X_j}\alpha(X_1,\ldots,Y,\ldots X_{k-1}),
  $$
  and we set $i_Y\alpha$ evaluated on elements of $C$ to vanish. This last argument can be actually summed up in the identity
  \begin{equation}
    \wedge^k V^* = E^* \otimes_{E\textrm{-sk}} \wedge^{k-1}V^* + \wedge^k C^*,
    \label{eq:k-product-V-E-C}
  \end{equation}
  where $E^* \otimes_{E\textrm{-sk}} \wedge^{k-1}V^*$ denotes the elements of $E^* \otimes \wedge^{k-1}V^*$ that are $E$-skew when seen as maps $E\to \wedge^{k-1}V^*$.
\end{proof}

The last result can be used to give a simple proof of the following characterization of lagrangian subspaces, see \cite[Lemma~A.1]{Za}:

\begin{lemma}\label{lem:Lagrangian sbspace}
  A subspace $L\subset V+ \wedge^k V^*$ is lagrangian if and only if $L$ is a standard isotropic subspace and $\Ann(E) = A_L$.
\end{lemma}

\begin{proof}
  If $L$ is lagrangian, we have already observed in (\ref{eq:lag}) that $\Ann(E)=A_L$ holds, and if $\dim(E)> n-k$, we have $L\cap \wedge^k V^* =\{0\}$, i.e., $A_L=\{0\}$. By Proposition~\ref{prop:inters1}, we must have $E=V$. So $\dim(E)\leq n-k$, or $\dim(E)=n$.

  For the converse, given $X+\alpha\in L^\perp$, we have
  $$
  X\in (L\cap \wedge^k V^*)^\circ= \Ann(E)^\circ = E
  $$
  by \eqref{eq:doubleAnn-standard}, so there is $\alpha' \in \wedge^k V^*$ such that $X+\alpha'\in L$. The difference $\alpha-\alpha'$ belongs to $L^\perp\cap \wedge^k V^*$, so is in $\Ann(E)=L\cap \wedge^k V^*$, and hence in $L$. It follows that $X+\alpha = (X + \alpha') + (\alpha-\alpha') \in L$.
\end{proof}

\begin{example}\label{ex:E}
  Note that for $E\subseteq V$ the subspace $L=E+\mathrm{Ann}(E)$ is such that $\pr_2(L)^\circ = \Ann(E)^\circ$ and 
  $L\cap V = E$.   So it is weakly lagrangian, i.e., $\pr_2(L)^\circ = L\cap V$, if and only if $\Ann(E)^\circ=E$, which says that $L$ is standard. It results from Lemma \ref{lem:Lagrangian sbspace} that $L$ must be lagrangian. On the other hand, for $A_L\subset \Ann(E)$ such that $A_L^\circ = E$, the subspace $E+A_L$ is weakly lagrangian.
\end{example}

\begin{remark} The previous lemma leads to an alternative characterization of standard isotropic subspaces of $V+ \wedge^k V^*$: an isotropic $L$ is standard if and only if there is a lagrangian subspace $L_0\supseteq L$ with $\pr_1(L_0)=\pr_1(L)$. Since any lagrangian is standard, such $L_0$ cannot exist if $L$ is not standard. Conversely, writing $L(E,A_L,\varepsilon)$ for a standard $L$, we see that $L_0=L(E,\Ann(E),\varepsilon_0)$, where $\varepsilon_0$ is the composition of $\varepsilon$ with the natural projection $\wedge^k V^*/A_L \to \wedge^k V^*/\Ann(E)$.
\end{remark}

Just as Poisson bivectors on $V$ are identified with lagrangian subspaces $L\subset V+ V^*$ such that $L\cap V=\{0\}$, we see (cf. Example~\ref{ex:lHP}) that higher Poisson structures on $V$ are precisely given by weakly lagrangian subspaces of $V+ \wedge^k V^*$ intersecting $V$ trivially.

\begin{proposition}\label{prop:HP}
  Let $L\subseteq V+ \wedge^k V^*$ be weakly lagrangian. Then
  $L\cap V=\{0\}$ if and only if $L$ is the graph of a higher Poisson structure $(S,\Lambda)$.
  In this case, $L$  is lagrangian if and only if it is standard (i.e., $\mathrm{rank}(\Lambda)=n$ or $\mathrm{rank}(\Lambda)\leq n-k$), and $\Ann(\Ima(\Lambda))= \ker (\Lambda)$.
\end{proposition}

\begin{proof}
  This is basically explained in Example~\ref{ex:lHP}: a subspace $L\subset V+ \wedge^k V^*$
  satisfies $L\cap V=\{0\}$ if and only if it is the graph of a map $\Lambda:\pp(L)\to V$. Let $S=\pp(L)$.
  Then $L$ is isotropic if and only if $i_{\Lambda(\alpha)}\beta= - i_{\Lambda(\beta)}\alpha$ for all $\alpha,\beta \in S$, and $L$ is weakly lagrangian if and only if $S^\circ=\{0\}$. The assertion about the lagrangian case is a direct consequence of Lemma~\ref{lem:Lagrangian sbspace}.
\end{proof}

Recall that a non-degenerate form $\omega\in \wedge^{k+1}V^*$ defines a higher Poisson structure on $V$  with $L$ lagrangian (see Example~\ref{ex:form}) and $\mathrm{rank}(\Lambda)=n$. By considering direct products, one finds examples satisfying the other possible rank condition:

\begin{example}
  Let $V_i$ be a $n_i$-dimensional vector space, $i=1,2$, and let $V=V_1\times V_2$ and $n=n_1+n_2$.
  Consider the higher Poisson structures (of order $k$) on $V_1$ and $V_2$ defined by a non-degenerate form $\omega\in \wedge^{k+1}V_1^*$ and by $L=S=\wedge^k V^*_2$, respectively. Then their direct product defines a higher Poisson structure on $V$ satisfying  $\mathrm{rank}(\Lambda)=n_1 \leq n_1+n_2-k = n -k$.
\end{example}


\begin{example}\label{ex:w12}
  Let $V_i$ be a vector space of dimension $n_i$, and let $\omega_i\in \wedge^{k_i+1}V_i^*$ be non-degenerate, $i=1,2$. If $V=V_1\times V_2$, we may regard $\omega_i\in \wedge^{k_i+1}V^*$ via pullback by the natural projections $V\to V_i$. Then
  $$
  L=\{X + (i_X\omega_1)\wedge \omega_2\,|\, X\in V_1\} \subset V + \wedge^{k_1+k_2+1} V^*
  $$
  is a weakly lagrangian subspace (see \cite{BCI}) such that $L\cap \wedge^{k_1+k_2+1} V^*=\{0\}$ and $L\cap V=\{0\}$. So $L$ defines
  a higher Poisson structure $(S,\Lambda)$ on $V$, with $\ker(\Lambda)=\{0\}$. Then:
  \begin{itemize}
  \item if $n_2\geq k_1+k_2+1$, then $L$ is standard, but $$
    \Ann(\Ima(\Lambda))=\wedge^{k_1+k_2+1}V_2^*\neq\{0\}=\ker(\Lambda).
    $$
  \item if $n_2< k_1+k_2+1$, then $L$ is not standard, but
    $$
    \Ann(\Ima(\Lambda))=\wedge^{k_1+k_2+1}V_2^*=\{0\}=\ker(\Lambda).
    $$
  \end{itemize}
  This shows that the conditions in Proposition~\ref{prop:HP} characterizing higher Poisson structures defined by lagrangian subbundles
  are independent.
\end{example}

\subsection{$B$-field transforms}
Any linear isomorphism $\phi: V\to V$ gives rise to a pairing-preserving automorphism $\phi + (\phi^{-1})^*: V+\wedge^k V^* \to V+\wedge^k V^*$ which preserves weakly lagrangian subspaces. In generalized geometry, one is often interested in another type of symmetry, that we now briefly discuss.

Any $B\in \wedge^2 V^*$ gives rise to a pairing-preserving automorphism of $V+ V^*$, referred to as a {\em gauge transformation} or {\em $B$-field transform} \cite{Se} (see also \cite{Gua,Hi}) via
$$
X+\alpha \mapsto e^B(X+\alpha)=X + \alpha + i_XB.
$$
It follows that for any lagrangian (resp. isotropic) subspace $L\subset V+ V^*$,
$$
e^B L = \{ X+\alpha+i_XB \st X+\alpha \in L \}
$$
is again lagrangian (resp. isotropic). Analogously, for $B\in \wedge^{k+1}V^*$, the same formula above defines a pairing-preserving
automorphism $e^B: V+ \wedge^kV^* \to V+ \wedge^kV^*$ preserving lagrangian (resp. isotropic) subspaces.
However, in contrast with lagrangian and isotropic subspaces, weakly lagrangian subspaces are {\em not} preserved by $B$-field transforms:

\begin{example}
  Let $V=V_1\times V_2$  and $L=\{X+ (i_X\w_1)\wedge \w_2|X\in V_1\} \subset V+ \wedge^k V^*$ be the  weakly lagrangian subspace considered in Example~\ref{ex:w12}. For $B=-\w_1\wedge \w_2 \in \wedge^2 V^*$ we get that $e^BL=V_1\times \{0\} \subseteq V$, which is not weakly lagrangian in $V+ \wedge^k V^*$ (as long as $V_2\neq \{0\}$).
\end{example}

Given $L\subset V+ \wedge^k V^*$ weakly lagrangian and $B\in \wedge^{k+1} V^*$, for $e^BL$ to be weakly lagrangian one must verify that $e^B L\cap V = (e^B L)^\perp \cap V$, which one can check to be equivalent to
the condition
\begin{equation}\label{eq:grB}
  L \cap \mathrm{gr}(-B) = L^\perp \cap \mathrm{gr}(-B),
\end{equation}
where $\mathrm{gr}(B)=\{X + i_XB\,|\, X\in V\}$.
We collect properties of $B$-field transforms of weakly lagrangian subspaces in the next result.

\begin{proposition}\
  \begin{itemize}
  \item[(a)] The $B$-transform of a weakly lagrangian subspace of the form $L \subset \wedge^k V^*$ is always weakly lagrangian.

  \item[(b)] The $B$-transform of a standard weakly lagrangian subspace $L\subset V+ \wedge^k V^*$ with $E\neq \{0\}$ is weakly lagrangian for every $B$ if and only if $L$ is lagrangian.

  \item[(c)] For $L\subset V+ \wedge^k V^*$ a non-standard weakly lagrangian subspace there always exists a $B$-transform such that $e^{B}L$ is not weakly lagrangian.
  \end{itemize}
\end{proposition}

\begin{proof}
  For (a), note that for a weakly lagrangian $L\subset \wedge^k V^*$ we have $L^\perp = L^\circ +\wedge^k V^*$. Since $L^\perp\cap V = L^\circ=\{0\}$, we have $L^\perp = \wedge^k V^*$. Thus $L \cap \mathrm{gr}(B) = L^\perp \cap \mathrm{gr}(B)= \{0\}$ for all $B\in \wedge^{k+1} V^*$.

  To verify (b), consider a weakly lagrangian $L$ with non-trivial projection $E$ on $V$. Assume that $L^\perp\cap \mathrm{gr}(B)= L\cap \mathrm{gr}(B)$ for all $B$. We first note that, as a consequence, $L^\perp\cap \wedge^k V^* =  L \cap \wedge^k V^*$: indeed, given $\alpha \in L^\perp\cap \wedge^k V^*$, take $Y+\beta\in L\subset L^\perp$ with $Y\neq 0$; then the element $Y+\beta+\alpha$ is null, so it belongs to $\mathrm{gr}(B)$ for some $B$, and is hence in $L^\perp\cap \mathrm{gr}(B)=L\cap \mathrm{gr}(B)$. Thus, $\alpha \in L\cap \wedge^k V^*$. Using this fact and \eqref{eq:doubleAnn-standard}, we see that
  $$
  (L\cap \wedge^k V^*)^\circ=(L^\perp \cap \wedge^k V^*)^\circ = \Ann(E)^\circ= E,
  $$
  so $L$ is lagrangian as a result of Lemma~\ref{lem:Lagrangian sbspace}.

  For the claim in (c), we know that $L$ is of the form $L(E,0,\varepsilon)$ with $E\neq V$ and $\varepsilon:E\to \wedge^k V^*$ such that $\iota_E\circ \varepsilon=0$. Take any skew-symmetric extension $B\in \wedge^{k+1}V^*$ of $-\varepsilon$. We then have $e^{B}L=E$, which is not weakly lagrangian.
\end{proof}

\section{Integrability}\label{sec:involution}


On any manifold $M$,  the vector bundle $TM+ \wedge^k\TM$ carries a $\wedge^{k-1}T^*M$-valued symmetric pairing defined pointwise by  \eqref{eq:pairing}. The notions of isotropic, lagrangian or weakly lagrangian subbundle of  $TM+ \wedge^k\TM$ are also defined pointwise.

Just as usual Dirac structures involve an integrability condition with respect to the Courant bracket \cite{Co}, we will consider the higher-order analogue of the Courant-Dorfman bracket on
$\Gamma(TM+\wedge^k \TM)$ given by
\begin{equation*}
  \ca X+\alpha,Y+\beta \cc=[X,Y]+\Lie_X\beta -i_Yd\alpha.
\end{equation*}
An isotropic subbundle $L\subset TM+\wedge^k T^*M$ is said to be {\em integrable}  if $\Gamma(L)$ is involutive with respect to this bracket. In this case, the vector bundle $L\to M$ inherits a Lie algebroid structure, with anchor given by the projection $\pr_1|_L: L\to TM$, and Lie bracket on $\Gamma(L)$ given by the restriction of $\ca\cdot,\cdot \cc$. It follows that $M$ has a singular foliation by ``orbits'' of this Lie algebroid, defined by integral leaves of $\pr_1(L)\subseteq TM$.

For ordinary Dirac structures, it is well known that each leaf carries a presymplectic structure, and that the resulting presymplectic foliation completely determines the Dirac structure. In this section, we will present an extension of this result for arbitrary integrable, isotropic subbundles $L\subset TM+\wedge^k T^*M$. We characterize their leafwise geometry and apply the results to higher Dirac structures.

Another feature resulting from the integrability of Dirac structures is the presence of a natural Poisson algebra of ``admissible functions'' on any Dirac manifold. This was extended to the context of isotropic, involutive subbundles of $TM+\wedge^k T^*M$ in \cite{Za}, where ``admissible functions'' are shown to define higher Lie algebras, as in \cite{CR} for multisymplectic structures.


\subsection{An equivalent viewpoint to involutivity}
In order to motivate the upcoming definitions, let  $L\subset TM+ \wedge^k\TM$ be an isotropic subbundle.
We keep the notation $E=\pr_1(L)\subseteq TM$ and $A_L=L\cap\wedge^k T^*M$, which are now general distributions (not necessarily of constant rank). Assuming that $L$ is integrable, a first consequence is the integrability of $E$. So let $\mathcal{O}\hookrightarrow M$ be a leaf integrating $E$, i.e., $E|_{\mathcal{O}}=T\mathcal{O}$. The restriction $A_L|_{\mathcal{O}} $ has constant rank, so  it is a subbundle of $\wedge^k T^*M|_{\mathcal{O}}$. We denote by
\begin{equation}\label{eq:FO}
  F_{\mathcal{O}}:= \frac{\wedge^k T^*M|_{\mathcal{O}}}{A_L|_\mathcal{O}} \to \mathcal{O}
\end{equation}
the resulting quotient bundle. Using  \eqref{eq:def-epsilon} pointwise, we obtain a section
\begin{equation}\label{eq:epsO}
  \varepsilon_{\mathcal{O}} \in \Gamma(T^*\mathcal{O} \otimes F_\mathcal{O}),
\end{equation}
with the additional property that $\iota_{T\mathcal{O}}\circ \varepsilon_{\mathcal{O}}\in \Gamma(\wedge^2 T^*\mathcal{O}\otimes \wedge^{k-1}\TM|_{\mathcal{O}} )$, where
\begin{equation}
  \iota_{T\mathcal{O}}: F_\mathcal{O} \to T^*\mathcal{O}\otimes \wedge^{k-1}T^*M|_{\mathcal{O}}\label{eq:iota-TO}
\end{equation}
is defined pointwise as in \eqref{eq:iotaE}.

Recall that, although the exterior differential $d$ is not well-defined on $\Gamma(\wedge^\bullet T^*M|_{\leaf})$, for $X\in \Gamma(T\leaf)$ we do have operators
\begin{equation}
  \Lie_X, i_Xd, di_X : \Gamma(\wedge^\bullet T^*M|_{\leaf})\to \Gamma(\wedge^\bullet T^*M|_{\leaf}),\label{eq:operators-in-forms}
\end{equation}
given by $\Lie_X(\alpha)=(\Lie_{\widetilde{X}}\widetilde{\alpha})|_{\leaf}$, $i_Xd(\alpha)= (i_{\widetilde{X}}d\widetilde{\alpha})|_{\leaf}$ and $di_X(\alpha)=(di_{\widetilde{X}}\widetilde{\alpha})|_{\leaf}$, where $\widetilde{X} \in \Gamma(TM)$ and $\widetilde{\alpha}\in \Gamma(\wedge^\bullet T^*M)$ are any extensions of $X$ and $\alpha$.

 A second consequence of the integrability of $L$ is that, for $X\in \Gamma(T\leaf)$, we have restrictions of the maps in (\ref{eq:operators-in-forms}) to
\begin{equation*}\label{eq:mod}
  \Lie_X,i_Xd,di_X:\Gamma(A_L|_{\leaf})\to \Gamma(A_L|_{\leaf}).
\end{equation*}
Indeed, for extensions $\widetilde{\alpha} \in \Gamma(A_L)$ and $\widetilde{X}+\widetilde{\gamma} \in \Gamma(L)$,
\begin{equation*}\label{eq:welldef}
  \ca \widetilde{X}+\widetilde{\gamma}, 0+\widetilde{\alpha} \cc|_\leaf = (\Lie_{\widetilde{X}}\widetilde{\alpha})|_{\leaf} = (i_{\widetilde{X}}d\widetilde{\alpha})|_{\leaf}   \in \Gamma(A_L|_\leaf),
\end{equation*}
whereas $di_X$ sends $\Gamma(A_L|_{\leaf})$ to zero.  As a result, the operators in \eqref{eq:operators-in-forms} descend to 
\begin{equation}\label{eq:iX}
  \Lie_X,i_Xd,di_X: \Gamma(F_{\mathcal{O}})\to \Gamma(F_{\mathcal{O}}).
\end{equation}

\begin{remark}\label{rem:cartanid}
  It is clear that the usual identities involving $\Lie_X$, $i_Xd$ and $di_X$
  hold for operators on $\Gamma(\wedge^\bullet T^*M|_{\leaf})$ and $\Gamma(F_\mathcal{O})$: for example, for $X, Y\in \Gamma(T\mathcal{O})$,
  \begin{equation}\label{eq:identities}
    \Lie_X=di_X+i_Xd,\;\;\; i_{[X,Y]}d=\Lie_X i_Yd - i_Yd \Lie_X,\;\;\; di_{[X,Y]}= \Lie_X di_Y - di_Y \Lie_X.
  \end{equation}
\end{remark}

As a third consequence of the integrability of $L$, by using the description of Proposition \ref{prop:isotropic subsp}, we have the following identity in $\Gamma(F_\leaf)$. For all $X,Y\in \Gamma(T\mathcal{O})$,
\begin{equation*}
  \varepsilon_{\mathcal{O}}([X,Y])=\Lie_X (\varepsilon_{\mathcal{O}}(Y)) - i_Y d (\varepsilon_{\mathcal{O}}(X)).
\end{equation*}

It follows from Proposition~\ref{prop:isotropic subsp} that, at each point of $\mathcal{O}$, the subbundle $L$ can be reconstructed from $E|_{\leaf}$, $A_{L}|_\leaf$ and $\varepsilon_{\mathcal{O}}$.  In fact, one can now verify that the integrability condition for $L$ is equivalent to the three conditions we have described:

\begin{theorem}\label{thm:integrability}
  An isotropic subbundle $L$ is integrable if and only if
  \begin{itemize}
  \item[(a)] the distribution $E=\pr_1(L)$ is integrable,
  \item[(b)] for each leaf $\leaf$, the space $\Gamma(A_L|_\leaf)$ is invariant by $\Lie_X$ for $X\in \Gamma(T\leaf)$,
  \item[(c)] for each leaf $\mathcal{O}$, we have, for all $X,Y\in \Gamma(T\mathcal{O})$,
    \begin{equation}\label{eq:cocycle}
      \Lie_X (\varepsilon_{\mathcal{O}}(Y)) - i_Y d (\varepsilon_{\mathcal{O}}(X)) - \varepsilon_{\mathcal{O}}([X,Y])=0\in \Gamma(F_\leaf).
    \end{equation}
  \end{itemize}
\end{theorem}
\begin{proof}
  We have already seen that the integrability of $L$ implies (a), (b) and (c). Conversely, given (a), the bracket of two sections of $L$ on a point $x$ in a leaf $\leaf$ only depends on the value along  $\leaf$. By (b), \eqref{eq:cocycle} is well defined and the integrability of $L$ is a consequence of (a) and (c).
\end{proof}

Note that from (a), we always have, for $X\in \Gamma(T\leaf)$,
\begin{equation}
  \Lie_X(\Gamma(\Ann(T\leaf)))\subseteq \Gamma(\Ann(T\leaf)),\label{eq:lieX-annE}
\end{equation}
so condition (b) is automatically satisfied when $A_L=\Ann(E)$.

The following are examples of integrable isotropic subbundles with $\varepsilon_\leaf=0$.

\begin{example}
  Consider an isotropic subbundle of the form $L=E+ \mathrm{Ann}(E)$ for a subbundle $E\subseteq TM$ (cf. Example~\ref{ex:E}). Condition (a) means that $E$ is integrable, whereas condition (b) follows from (a), and (c) is trivially satisfied. Therefore, $L$ is integrable if and only if $E$ is integrable.
\end{example}

\begin{example}
  In the case of a subbundle $S\subseteq \wedge^kT^*M$ (cf. Example~\ref{ex:lHP1}), consider the isotropic subbundle $L=S^\circ+S$. Condition (b) amounts to $\Lie_X \alpha\in \Gamma(S)$ for all $X\in \Gamma(S^\circ)$ and $\alpha\in \Gamma(S)$. Condition (a) then follows from (b), and (c) is trivially satisfied.
\end{example}

Our next goal is to identify the differential complex with respect to which \eqref{eq:cocycle} is a cocycle condition.

\subsection{Differential complex on leaves}\label{sec:diffcomplex}

Let $L\subset TM+ \wedge^k\TM$ be an isotropic subbundle. Throughout this section, we will suppose that
\begin{itemize}
\item the distribution $E\subseteq TM$ is integrable,
\item for each leaf $\leaf$, the space $\Gamma(A_L|_{\leaf})$ is invariant by $\Lie_X$ for $X\in \Gamma(T\leaf)$.
\end{itemize}
As previously observed, both conditions hold whenever $L$ is integrable. Let
$\mathcal{O}\hookrightarrow M$ be a leaf of $E$, and
let $F_{\mathcal{O}} \to \mathcal{O}$ be as in \eqref{eq:FO}. The second assumption above guarantees that the operators in \eqref{eq:iX} are well defined.

Denote by $\Omega^p(\leaf,F_{\mathcal{O}})=\Gamma(\wedge^pT^*\mathcal{O}\otimes F_\mathcal{O})$ the space of $p$-forms on $\mathcal{O}$ with values in $F_\mathcal{O}$. We say that $\eta \in \Omega^p(\leaf,F_{\mathcal{O}})$ is {\em $T\mathcal{O}$-skew} if
\begin{equation}\label{eq:TOskew}
  \iota_{T\leaf}\circ \eta \in \Gamma(\wedge^{p+1} T^*\leaf \otimes \wedge^{k-1} T^*M|_{\mathcal{O}}),
\end{equation}
where $\iota_{T\leaf}$ is defined in \eqref{eq:iota-TO}, and denote by
$$
\Osk{p} \subseteq \Omega^p(\mathcal{O},F_{\mathcal{O}})
$$
the subspace of $T\mathcal{O}$-skew $p$-forms. In what follows, we will show that, for each $p$, there is a natural operator
\begin{equation*}\label{eq:delta}
  \delta: \Osk{p} \to \Osk{p+1}
\end{equation*}
making $\Oskb$ into a differential complex. As a warm up, we see what $\delta$ looks like in degrees $0$ and $1$.

For $p=0$, let $\delta: \Gamma(F_\mathcal{O})=\Osk{0} \to \Gamma(T^*\mathcal{O}\otimes F_\mathcal{O})$ be given by
$$
\delta \theta (X) := i_Xd \theta,
$$
for $X\in \Gamma(T\mathcal{O})$ (cf. \eqref{eq:iX}). It is a direct verification that $\delta\theta \in \Osk{1}$, hence
$$
\delta: \Osk{0} \to \Osk{1}.
$$

For $p=1$, we have the following results.

\begin{proposition}\label{prop:delta-p=1}
  For $\theta \in \Gamma(T^*\mathcal{O}\otimes F_{\mathcal{O}})$, the expression
  (cf. \eqref{eq:cocycle})
  $$
  \delta\theta(X,Y) =  \Lie_X (\theta(Y)) - i_Y d (\theta(X)) -  \theta([X,Y])
  $$
  defines an operator
  $\delta: \Osk{1} \to \Osk{2},$
  such that $\delta\circ \delta:  \Osk{0} \to \Osk{2}$ vanishes.
\end{proposition}

\begin{proof}
  First, $\delta\theta$ defines an element in $\Gamma(\wedge^2 T^*\mathcal{O}\otimes F_\mathcal{O})$ since, for $X,Y\in \Gamma(T\mathcal{O})$,
  $$
  \delta\theta(X,Y)+\delta\theta(Y,X)=di_X (\theta(Y))+di_Y(\theta(X))=0.
  $$
  On the other hand, for $X,Y\in \Gamma(T\mathcal{O})$, we have the identity
  \begin{equation}\label{eq:identity-Lie}
    [\Lie_X,i_Y]=[i_X,\Lie_Y]=i_{[X,Y]}
  \end{equation}
  of operators $\Gamma(F_{\mathcal{O}})\to \Gamma(\wedge^{k-1}T^*M|_{\mathcal{O}})$.
  Using this identity and that $\theta$ is $T\mathcal{O}$-skew, we have
  \begin{align*}
    i_Z(\delta \theta(X,Y)) & {} = i_Z\Lie_X\theta(Y) - i_Zi_Yd\theta(X)-i_Z\theta([X,Y]) \\
    & {} = i_{[Z,X]} \theta(Y) + \Lie_X i_Z\theta(Y) +i_Yi_Zd\theta(X)+i_{[X,Y]}\theta(Z)\\
    & {} = i_Y\theta([X,Z])+i_Yi_Zd\theta(X) - i_Y\Lie_X\theta(Z) = -(\delta\theta)(X,Z)(Y),
  \end{align*}
  i.e., $\delta\theta$ is in $\Osk{2}$. Finally, for $\theta\in \Osk{0}$,
  $$\delta(\delta \theta)(X,Y) = \Lie_X(i_Y d\theta)-i_Yd(i_Xd\theta)-i_{[X,Y]}d\theta,$$
  which vanishes by \eqref{eq:identities}.
\end{proof}

We will extend the last proposition to higher degrees.
Note that $\delta$ can be written, for $p=0,1$ and  $\theta_p\in \Osk{p}$, as
\begin{align*}
  \delta\theta_0(X) & {} =\Lie_X \theta_0 - di_X\theta_0,\\
  \delta\theta_1(X,Y) & {} = \Lie_X
  \theta_1(Y)-\Lie_Y\theta_1(X)-\theta_1([X,Y]) + di_Y\theta_1(X).
\end{align*}
These expressions and the usual Koszul formula suggest the following definition.

\begin{definition}\label{def:delta}
  For $\theta\in \Osk{p}$, define $\delta\theta \in
  \Gamma(\wedge^{p+1}T^*\leaf \otimes F_\leaf)$ by
   \begin{align}\label{eq:def-delta}
  \delta\theta(X_0,\ldots, X_p) =  {} &  \sum_{j\leq p} (-1)^j \Lie_{X_j} \theta(X_0, \ldots, \widehat{X}_j, \ldots, X_p) \nonumber \\ & + \sum_{j<l} (-1)^{j+l}\theta([X_j,X_l], \ldots,  \widehat{X}_j,\ldots,\widehat{X}_l, \ldots, X_p)   \\
  & -(-1)^p d i_{X_p} \theta(X_0, \ldots, X_{p-1}).\nonumber
  \end{align}
\end{definition}
Although similar to the usual Koszul formula, this definition strongly depends on the existence of the operators $\Lie_X$, $di_X$ on $\Gamma(F_\leaf)$ defined in \eqref{eq:iX}. Note that the last term, which would vanish in the formula for the usual exterior derivate, does not vanish in our case.

Let us verify that the image of $\delta$ is $T\leaf$-skew.

\begin{lemma}
  For $\theta\in\Osk{p}$, we have that $\delta\theta\in\Osk{p+1}$.
\end{lemma}
\begin{proof}
  It is enough to prove that, for $Z\in \Gamma(T\leaf)$,
  $$
  i_Z\delta\theta(X_0,\ldots, X_{p-1},X_p)=-i_{X_p} \delta\theta(X_0,\ldots, X_{p-1},Z).
  $$
  First, we use the following identities on $\Gamma(F_\leaf)$. On $\theta(X_1,\ldots, X_{p})$ we apply
  \begin{align*}
    i_Z \Lie_{X_p} - i_Z di_{X_p} & {} = i_Z i_{X_p} d= -i_{X_p}i_Z d = - i_{X_p} \Lie_{Z} + i_{X_p} di_{Z},
  \end{align*}
  and on $\theta(X_1,\ldots, \widehat{X}_j,\ldots, X_{p-1})$, for $j<p$, we apply
  \begin{align*}
    i_Z\Lie_{X_j}i_{X_p} & {} = \Lie_{X_j}i_Zi_{X_p} - i_{[X_j,Z]}i_{X_p}= -\Lie_{X_j}i_{X_p}i_Z  - i_{[X_j,Z]}i_{X_p}\\
    & {}  = -i_{X_p}\Lie_{X_j}i_Z - i_{[X_j,X_p]}i_Z - i_{[X_j,Z]}i_{X_p}.
  \end{align*}
  The result then follows from
  $$
  (-1)^{j+p} i_Z\theta([X_j,X_p],X_0, \ldots,  \widehat{X}_j, \ldots, X_{p-1}) = (-1)^j i_{[X_j,X_p]}i_Z\theta(X_0,\ldots,  \widehat{X}_j,\ldots, X_{p-1}).
  $$

\end{proof}

The only difference between \eqref{eq:def-delta} and the expression for the usual exterior derivative is the last term, which we describe as follows.

\begin{lemma}\label{lemma:ddeltaomega}
  For $\theta\in \Osk{p-1}$ and $X_0,\ldots,X_p\in \Gamma(T\leaf)$,
 \begin{align*}
di_{X_p}(\delta\theta(X_0,\ldots, X_{p-1})) =  {} & \sum_{j<p} (-1)^j
\Lie_{X_j} di_{X_p}\theta(X_0,\ldots,\widehat{X}_j,\ldots, X_{p-1})\\ & +
(-1)^p \Lie_{X_p}
di_{X_{p-1}}\theta(X_0,\ldots, X_{p-2})\\ & +
\sum_{j<l} (-1)^{j+l}(-1)^{p-1} d i_{[X_j,X_l]} \theta(X_0,\ldots,\widehat{X}_j,\ldots,\widehat{X}_l,\ldots,X_p).
\end{align*}
\end{lemma}

\begin{proof}
  Develop $\delta \theta $, use the identity $di_{X_p} \Lie_{X_j} = \Lie_{X_j}di_{X_p}-di_{[X_j,X_p]}$  on $\Gamma(F_\leaf)$ (see \eqref{eq:identities}), together with the fact that $\theta$ is $T\leaf$-skew.
\end{proof}

We are now ready to prove that $\delta$ defines a differential.

\begin{lemma}\label{lemma:squares-to-zero}
  The operator $\delta:\Osk{\bullet} \to  \Osk{\bullet+1}$ squares to zero.
\end{lemma}
\begin{proof}
  Let $\omega\in\Osk{p-1}$ and consider $(\delta(\delta\omega))(X_0,\ldots,X_p)$. The combinations of the terms in \eqref{eq:def-delta} that appear in the usual definition of the exterior derivative add up to zero. The remaining terms are $-(-1)^{p-1}$ times
  \begin{multline*}
  \sum_{j<p} (-1)^j \Lie_{X_j} di_{X_p}\omega(X_0,\ldots,\widehat{X}_j,\ldots,X_{p-1}) +
  (-1)^p \Lie_{X_p} di_{X_{p-1}}\omega(X_0,\ldots,{X}_{p-2})\\
  + \sum_{j<l} (-1)^{j+l}  d i_{X_p} \omega([X_j,X_l],\ldots,\widehat{X}_j,\ldots,\widehat{X}_l,\ldots, X_{p-1}) \\
  +\sum_{j<p}(-1)^{j+p} di_{X_{p-1}}\omega([X_j,X_p],\ldots, \widehat{X}_j,\ldots, {X}_{p-2}).
  \end{multline*}
  By using the $T\leaf$-skewness of $\omega$, the last two terms add up to
  $$
  (-1)^{p-1}\sum_{j<l} (-1)^{j+l}di_{[X_j,X_l]} \omega(X_0,\ldots, \widehat{X}_j,\ldots,\widehat{X}_l, \ldots, X_{p}).
  $$
  By using Lemma \ref{lemma:ddeltaomega}, we see that $\delta(\delta\omega)$ vanishes.
\end{proof}

Hence, for each leaf $\mathcal{O}$ of $L$,  $(\Osk{p},\delta)$ is a differential complex. We now observe how it relates to $(\Omega^\bullet(\mathcal{O}),d)$, the usual complex of differential forms on $\mathcal{O}$.

For $\theta\in\Osk{p}$ and $Y,Z\in\Gamma(T\leaf)$ we have
$$
i_Y i_Z \theta(X_1,\ldots, X_p) = -i_Z i_Y \theta(X_1,\ldots, X_p),
$$
where the first insertion is a map $\Gamma(F_\leaf)\to \Gamma(\wedge^{k-1}T^*M|_{\leaf})$ whereas the second is a map $\Gamma(\wedge^{k-1}T^*M|_{\leaf}) \to \Gamma(\wedge^{k-2}T^*M|_{\leaf})$. This means that the restriction of $\theta$ is skew-symmetric and we thus have a restriction map
\begin{equation}\label{eq:rO}
  r_\leaf:\Oskb \to \Omega^{\bullet+ k}(\leaf).
\end{equation}
The following theorem consolidates the results of this section.

\begin{theorem}\label{theo:delta}
  The operator $\delta$ defined in \eqref{eq:def-delta} is a differential on $\Oskb$, making it into a complex,
  and the restriction map $r_{\leaf}: (\Oskb,\delta) \to (\Omega^{\bullet+ k}(\leaf),d)$ is a chain map.
\end{theorem}

\begin{proof}
  It remains to check the last assertion about the restriction map being a chain map.
  So take $X_0, \ldots, X_{p+k}\in \Gamma(T\leaf)$. We first consider $i_{X_{p+1}} \delta\theta(X_{0},\ldots, X_{p})$. Using, as before, $\Lie_{X_{p+1}}=di_{X_{p+1}}+i_{X_{p+1}}d$ on $\Gamma(F_\leaf)$ and the identity $[\Lie_{X_j},i_{X_{p+1}}]=i_{[X_j,X_{p+1}]}$ for operators $\Gamma(F_{\mathcal{O}})\to \Gamma(\wedge^{k-1}T^*M|_{\mathcal{O}})$, we have
  \begin{align*}
    i_{X_{p+1}} \delta\theta(X_{0},\ldots, X_{p}) & {} = \sum_{j\leq p+1} (-1)^j
    \Lie_{X_j} (\iota_{T\leaf}\circ \theta)(X_0, \ldots, \widehat{X}_j, \ldots, X_{p+1})  \\
    & \phantom{ {} = } + \sum_{j<l\leq p+1}
    (-1)^{j+l}(\iota_{T\leaf}\circ \theta)([X_j,X_l] \ldots  \widehat{X}_j,\ldots,\widehat{X}_l, \ldots, X_{p+1}) \\
    & \phantom{ {} = }  -(-1)^{p+1}
    d i_{X_{p+1}}  (\iota_{T\leaf}\circ \theta)(X_0, \ldots, X_{p}).\\
  \end{align*}
  Applying the same identities on $\Gamma(\wedge^{k-r} T^*M|_{\leaf})$ for $i_{X_{p+r}}\ldots i_{X_{p+1}}\delta\theta(X_{0},\ldots, X_{p})$, with $0\leq r\leq k$, we get to the expression for $(r_\leaf(\delta \theta))(X_0,\ldots, X_{p+k})$:
 \begin{align*}
i_{X_{p+k}}\ldots i_{X_{p+1}} \delta\theta(X_{0},&\ldots, X_{p}) \\ & {} = \sum_{j\leq p+k} (-1)^j
\Lie_{X_j} (r_\leaf \theta)(X_0, \ldots, \widehat{X}_j, \ldots, X_{p+k})  \\
& \phantom{ {} = } + \sum_{j<l\leq p+k}
(-1)^{j+l}(r_\leaf \theta)([X_j,X_l], \ldots,  \widehat{X}_j,\ldots,\widehat{X}_l, \ldots, X_{p+k}) \\
& {} = (d(r_\leaf \theta))(X_0,\ldots, X_{p+k}),
\end{align*}
  where we have used that  $d i_{X_{p+k}}  (r_\leaf \theta)(X_0, \ldots, X_{p})=0$.
\end{proof}

The previous theorem  provides an interpretation of \eqref{eq:cocycle} as a cocycle condition.

\begin{corollary}\label{cor:condition-c}With the hypotheses of  Theorem~\ref{thm:integrability}, condition (c) is equivalent to
  \begin{itemize}
  \item[(c')] for each integral leaf $\mathcal{O}\hookrightarrow M$, the form $\varepsilon_\leaf \in \Osk{1}$ satisfies $\delta\varepsilon_\leaf=0$.
  \end{itemize}
\end{corollary}

Note that when $\leaf=M$ and $A_L=0$, the restriction map $r_\leaf$ is an isomorphism with the (shifted) de Rham complex of the manifold:
\begin{equation}\label{eq:rO-iso}
(\Omega_{\textrm{sk}}^\bullet(M,\wedge^kT^*M),\delta) \stackrel{\sim}{\to} (\Omega^{\bullet+k}(M),d).
\end{equation}
 On the other hand, when $k=1$ and $L\subseteq TM+T^*M$ is lagrangian, then $A_L|_\leaf=\Ann(T\leaf)$ and $F_{\mathcal{O}}\cong T^*\mathcal{O}$, and the restriction map  \eqref{eq:rO} is an isomorphism
\begin{equation}\label{eq:iso}
  (\Osk{\bullet},\delta) \stackrel{\sim}{\to} (\Omega^{\bullet+1}(\leaf),d).
\end{equation}

For an integrable, isotropic subbundle $L$,
it follows from Theorem~\ref{theo:delta} that, on each leaf, the 1-cocycle $\varepsilon_\leaf$ restricts to a closed $(k+1)$-form on $\leaf$,
$$
\omega_\leaf:=r_\leaf(\varepsilon_\leaf)\in \Omega^{k+1}(\leaf).
$$
Note that, while one can completely recover $L|_\leaf$ from $\varepsilon_\leaf$ (cf. Proposition~\ref{prop:isotropic subsp}),
one may lose a lot of information by passing to $\omega_\leaf$.

For usual Dirac structures in $TM+T^*M$, using \eqref{eq:iso} one sees that Theorem~\ref{thm:integrability} and Corollary~\ref{cor:condition-c}  describe their underlying presymplectic foliations, as in \cite[Thm.~2.3.6]{Co}. More generally, for $k=1$, the form $\omega_\leaf$ is a presymplectic form on the leaf, and the resulting presymplectic foliation is the one underlying the only Dirac structure $\widetilde{L}$ containing $L$ and satisfying $\pr_1(\widetilde{L})=\pr_1(L)$. The information about $L$ is encoded in the cocycle $\varepsilon_\leaf: T\leaf\to F_\leaf$, which can be regarded as a lift of the presymplectic structure $\omega_\leaf$ (seen as a map $T\leaf\to T^*\leaf$). So, even for $k=1$, the theorem above provides new information,
in that it extends \cite[Thm.~2.3.6]{Co}  to general integrable isotropic subbundles of $TM + T^*M$.

The complex $\Osk{\bullet}$ is the suitable framework to describe isotropic subbundles of $TM + \wedge^k T^*M$. However, its full meaning is yet to be understood and will be the focus of future work.

\medskip

\paragraph{\bf Regular case.}
Let us consider an isotropic subbundle $L\subset TM + \wedge^kT^*M$ such that the distribution $E=\pr_1(L)\subseteq TM$ is integrable and regular, i.e., has constant rank, and the space $\Gamma(A_L)$ is invariant by $\Lie_X$ for $X\in \Gamma(E)$. In this case, one can make sense of the leafwise constructions of Section~\ref{sec:diffcomplex} globally, over $M$.

The space $\Omega^\bullet_E(M):=\Gamma(\wedge^\bullet E^*)$ of ``foliated forms'' on $M$ is equipped with a natural differential $d_E$ (viewing $E$ as a Lie algebroid). Since $E$ is regular,
$A_L=L\cap\wedge^k T^*M$ is a subbundle of $\wedge^k T^*M$, and we consider the quotient bundle, cf. \eqref{eq:FO},
$$
F=\wedge^k T^*M/A_L \to M.
$$
In the space $\Omega^\bullet_E(M,F)=\Gamma(\wedge^\bullet E^*\otimes F)$ of forms with values in $F$, we have the subspace of {\em $E$-skew forms} (defined  as in \eqref{eq:TOskew}),
$$
\Omega^\bullet_{E\textnormal{-}\mathrm{sk}}(M,F) \subseteq \Omega^\bullet_E(M,F).
$$
Just as in Definition~\ref{def:delta}, one has a differential
$$
\delta_E: \Omega^\bullet_{E\textnormal{-}\mathrm{sk}}(M,F) \to \Omega^{\bullet+1}_{E\textnormal{-}\mathrm{sk}}(M,F),
$$
so $(\Omega^\bullet_{E\textnormal{-}\mathrm{sk}}(M,F),\delta_E)$ is a complex. One also has a restriction defining a chain map analogous to \eqref{eq:rO},
$$
r: (\Omega^\bullet_{E\textnormal{-}\mathrm{sk}}(M,F),\delta_E) \to (\Omega^{\bullet + k}_E(M),d_E).
$$
Similarly to \eqref{eq:epsO}, the subbundle $L$ is determined pointwise by an element
$$
\varepsilon \in \Omega^1_{E\textnormal{-}\mathrm{sk}}(M,F).
$$
The following extends Theorem~\ref{thm:integrability} and Corollary~\ref{cor:condition-c} in the regular case.

\begin{proposition}\label{prop:s-closeness}
  Let $L\subset TM+ \wedge^k T^*M$ be an isotropic subbundle such that $\Gamma(A_L)$ is invariant by $\Lie_X$ for $X\in \Gamma(E)$ and the distribution $E$ is  involutive and regular. For each leaf $\tau:\mathcal{O}\hookrightarrow M$, the following is a chain map diagram:
  \begin{center}
    \begin{tikzpicture}
      \matrix (m) [matrix of math nodes,row sep=2em,column sep=4em,minimum
      width=2em]
      { (\Omega^\bullet_{E\textnormal{-}\mathrm{sk}}(M,F), \delta_E) & (\Omega^\bullet_{\mathrm{sk}}(\leaf,F_\leaf), \delta) \\
        (\Omega^{\bullet+k}_E(M),d_E) & (\Omega^{\bullet+k}(\leaf),d). \\ };
      \path[-stealth]
      (m-1-1) edge node[above]{$\tau^*$} (m-1-2)
      (m-1-1) edge node[left]{$r$} (m-2-1)
      (m-2-1) edge node[above]{$\tau^*$} (m-2-2)
      (m-1-2) edge node[left]{$r_\leaf$} (m-2-2);
    \end{tikzpicture}
  \end{center}
  Moreover, $L$ is integrable if and only if $\delta_E \varepsilon=0$.
\end{proposition}


\subsection{Higher Dirac structures}\label{subsec:higherdirac}

We define a higher Dirac structure as follows.

\begin{definition}
  A subbundle $L\subset TM+ \wedge^kT^*M$ is called a {\em higher Dirac structure} if it is involutive and weakly lagrangian at each point.
\end{definition}

We use Theorem~\ref{thm:integrability} and \eqref{eq:weak-Lagrangian sbspace} to give a complete description of the leafwise geometry of higher Dirac and higher Poisson structures.

\begin{theorem}
  An involutive, isotropic subbundle $L\subset TM+\wedge^k T^*M$ is a higher Dirac structure if and only if, for each leaf $\leaf$, the cocycle $\varepsilon_\leaf\in \Osk{1}$ satisfies
\begin{equation}
  \varepsilon_\leaf(\mathrm{pr}_2(L|_\leaf)^\circ)=\{0\}.\label{eq:cond-higher-dirac}
\end{equation}
The subbundle $L$ is higher Poisson if and only if, in addition to \eqref{eq:cond-higher-dirac}, we have $L\cap TM=\{0\}$, i.e., $\ker(\varepsilon_\leaf)=\{0\}$ for each leaf $\leaf$.
\end{theorem}

\begin{proof}
The first claim follows directly from \eqref{eq:weak-Lagrangian sbspace}. For the second one, from Proposition~\ref{prop:HP}, we know that weakly lagrangian subbundles $L\subset TM+ \wedge^k T^*M$   satisfying $L\cap TM=\{0\}$ are given by graphs of bundle maps $\Lambda: S\to TM$, where $S\subseteq \wedge^k T^*M$ is a subbundle such that $S^\circ =\{0\}$ and $i_{\Lambda(\alpha)}\beta = -i_{\Lambda(\beta)}\alpha$ for all $\alpha, \beta \in S$. One can then verify (see \cite{BCI}) that $L$ is integrable if and only if the pair $(S, \Lambda)$ is a higher Poisson structure (i.e., it satisfies (c) in Definition~\ref{def:hpois}).
\end{proof}

This picture of the leafwise geometry of higher Poisson structures $(S,\Lambda)$ complements the discussion in \cite[Sec.~5]{BCI}: in this case we have $A_L=\ker(\Lambda)$, and each leaf $\leaf$ of the foliation integrating $E=\Lambda(S)$ is equipped with  a $\delta_E$-closed $E$-skew form $\varepsilon_\leaf:T\leaf\to F_\leaf,$ given by $\varepsilon_\leaf(\Lambda(\alpha))=\alpha+\ker(\Lambda)$, which is moreover {\em nondegenerate}, i.e., injective.
The next example illustrates that this nondegeneracy of $\varepsilon_\leaf$ for higher Poisson structures may be lost upon restriction to the $(k+1)$-form $\omega_\leaf = r_\leaf(\varepsilon_\leaf)\in \Omega^{k+1}(\leaf)$.

\begin{example}\label{ex:nondeg}
Following Example~\ref{ex:w12} (see \cite[Ex.~6]{BCI}), given
$k_i$-plectic manifolds $(M_i,\omega_i)$, $i=1,2$, we have a
higher Poisson structure on $M=M_1\times M_2$ given by
$$
L =\{X+ i_X\omega_1\wedge \omega_2 \;|\; X\in
TM_1\} \subset TM+ \wedge^{k_1+k_2+1}T^*M.
$$
Here we view $\omega_i$ as a $k_i$-form on $M$ via pullback by the natural
projection $M\to M_i$. In this example, $E=TM_1\subseteq TM$, $A_L=\{0\}$, and
$$
\varepsilon: E\to \wedge^{k_1+k_2+1}T^*M,\;\;\; X\mapsto (i_X\omega_1)\wedge \omega_2.
$$
The leaves $\leaf$ are of the form $M_1\times \{y\}$, for $y\in M_2$.
It is clear that $\ker(\varepsilon) =\{0\}$, but the leafwise $(k_1+k_2+2)$-form obtained by restriction is identically zero.
\end{example}

Note that, when $k=1$, condition \eqref{eq:cond-higher-dirac} yields that $A_L|_\leaf=\Ann(T\leaf)$, so $\varepsilon_\leaf$ defines a $2$-form on $\leaf$, which is closed by Theorem \ref{thm:integrability} and the isomorphism \eqref{eq:iso}. We thus obtain the known presymplectic foliation underlying a Dirac structure \cite{Co}, which is symplectic, $\ker(\varepsilon_\leaf)=0$, exactly when the Dirac structure is Poisson.

\begin{example}\label{ex:k+1-form}\label{ex:HD-mnfd}
Given a form $\w \in \Omega^{k+1}(M)$, its graph $L=\mathrm{gr}(\w)$  defines a lagrangian subbundle of $TM+ \wedge^kT^*M$ so that $L\cap \wedge ^k\TM=\{0\}$. In terms of Proposition \ref{prop:s-closeness}, the leaves are given by the connected components of $M$, so the maps $\tau^*$ are identities when restricted to each connected component; since $A_L=\{0\}$, the restrictions give rise to isomorphisms \eqref{eq:rO-iso}. So the subbundle $L$ is integrable if and only if $d\omega=0$. Hence the graph of a closed $k+1$-form defines a lagrangian higher Dirac structure. In particular, multisymplectic forms correspond to lagrangian higher Dirac structures intersecting both $TM$ and $\wedge^k \TM$ trivially.
More generally, a higher Dirac structure $L$ such that $L\cap \wedge ^k\TM=\{0\}$ is determined by an integrable subbundle $E$ and  $\varepsilon\in \Omega^1_{E\textnormal{-}\mathrm{sk}}(M,F)$, where $F=\wedge^k T^*M$, satisfying $\delta_E\varepsilon=0$ and $\Ima(\varepsilon)^\circ = \ker(\varepsilon)$.
\end{example}

\begin{example}
  Given a top-dimensional multivector $\pi\in \wedge^{\dim M} TM$, let $k=\dim M -1$. Its graph $$L=\textup{gr}(\pi)=\{ i_\alpha\pi + \alpha \,|\, \alpha \in \wedge^{k}T^*M\}$$
defines a lagrangian subbundle of $TM+\wedge^kT^*M$ which is always integrable (\cite[prop.~3.4]{Za}). The points where $\pi$ vanishes are leaves $\leaf$ with $T\leaf=\{0\}$, $A_L|_\leaf=\wedge^k T^*M$, $\varepsilon_\leaf=0$ and the complex $\Osk{\bullet}$ vanishes. Away from these points, we have open leaves where $T\leaf=TM|_\leaf$, $A_L|_\leaf=\{0\}$ and $\varepsilon_\leaf$ comes from the volume form that is the inverse of $\pi$ on $\leaf$. By \eqref{eq:rO-iso} the complex $\Osk{\bullet}$ is isomorphic to the usual de Rham complex on $\leaf$ (with a shift).
\end{example}

\begin{example}
Let $N$ and $Q$ be manifolds. Consider a closed nonzero form $\w\in \Omega^{k+1}(N)$ and a function $f\in C^\infty(Q)$. For each $q\in Q$, define $\w_q:=f(q)\w$ to get a family of $k+1$-forms on $N$ parametrized by $Q$. We use this family to define on the product $M=N\times Q$ a regular isotropic subbundle of $TM+\wedge^k T^*M$, for $k\leq \dim Q$, by setting:  $E=TN$,  $A_L \subseteq \wedge^kT^*Q$  such that  $A_L^\circ=TN$ (e.g. $A_L= \wedge^kT^*Q$), and $\varepsilon(X)=i_X\w_q+A_L$ for  $X\in E$. The  corresponding isotropic subbundle $L$ is weakly lagrangian since, at each point of $M$, we have
$$
\pr_2(L)^\circ=(\image(\w_q)+A_L)^\circ=\ker(\w_q)=\ker(\varepsilon).
$$
By Lemma \ref{lem:Lagrangian sbspace}, $L$ is lagrangian if and only if $A_L=\Ann(E)=\wedge^k T^*Q$. To study integrability, we use conditions (a), (b) and (c') from  Theorem~\ref{thm:integrability} and Corollary~\ref{cor:condition-c}. Condition (a) holds trivially, while (b) requires that $\Gamma(A_L)$ is invariant by $\Lie_X$, with $X\in \Gamma(TN)$ (which is automatically satisfied when $A_L=\Ann(E)$). Finally, (c') is equivalent to $df\wedge(i_Yi_X\w)\in \Gamma(A_L)$ for any $X,Y\in \Gamma(TN)$. For $k=1$, this condition is always satisfied, as $A_L^\circ=TN$ implies $A_L=T^*Q$, and we recover the Dirac structure associated with a smooth family of presymplectic structures (parametrized by $Q$). For $k>1$, the condition is satisfied if and only if $df=0$,  leading to a much more rigid picture: a family which is constant on the connected components of $Q$.
\end{example}

\begin{example}
  Let us consider the classical field theory framework described in Section~\ref{sec:hpoiss}, and the higher Poisson structure $(S_{red},\Lambda_{red})$ on $Z=M/\mathbb{R}$ with coordinates $(x_i,y_j,p_{lk})$.
  Its foliation is defined by the fibres of the natural projection $Z\to B$ (resulting from the composition $Z\to P\to B$),  $(x_i,y_j,p_{lk})\mapsto (x_i)$ in coordinates. The subbundle $A_L=\ker(\Lambda)$ is identified with $\wedge^m\TM$, the bundle $F$ equals $\wedge^m T^*Z/\wedge^m\TM$, and the map $\varepsilon:E\to F$ is given by
  $$
  \varepsilon\left(\frac{\partial}{\partial y_l}\right) = \alpha_l + \wedge^m\TM,\;\;\;
  \varepsilon\left(\frac{\partial}{\partial p_{li}}\right) = -\gamma_{li} + \wedge^m\TM.
  $$
\end{example}

\begin{remark}\label{rmk:restriction of verepsilon}
According to \cite[Thm.~3.12]{Za}, a regular lagrangian subbundle $L=L(E,A_L,\varepsilon)$ is integrable if and only if $E$ is an involutive subbundle and, for any lift $E\to\wedge^k T^*M$ of $\varepsilon$ with extension $\tilde{\varepsilon}\in \Omega^{k+1}(M)$, we have
  \begin{equation}\label{eq:condition3}
    d\tilde{\varepsilon}|_{\wedge^3 E^* \otimes \wedge^{k-1} T^*M} =0.
  \end{equation}
  However, this is not enough for weakly lagrangian subbundles, as the next example shows.  Consider $k=2$ and $M=\mathbb{R}^5$, with coordinates $\{x_1,\ldots,x_5\}$. Set $E=\textup{span}\{ \partial_{x_1},\partial_{x_2} \}$ and, for a given $A_L$, define
  $$\varepsilon = dx_1\otimes(x_4dx_2\wedge dx_{3}+A_L) - dx_2\otimes(x_4dx_1\wedge dx_{3} + A_L).$$
Any extension $\tilde{\varepsilon}$ satisfies \eqref{eq:condition3} as $\wedge^3 E^*=\{0\}$ by dimensional reasons. The integrability of $L$ would imply that
  $$\ca \partial_{x_1}+x_4 dx_2\wedge dx_3, \partial_{x_2} - x_4 dx_1\wedge dx_{3}\cc= -dx_3\wedge dx_{4}\in \Gamma(A_L).$$
  Set $A_L=\textup{span}\{ dx_3\wedge dx_5,dx_4\wedge dx_5 \}\subset \Ann(E)$, so that $L=L(E,A_L,\ve)$ is not integrable. Finally, we check that $L$ is weakly lagrangian:
  $$\pp(L)^\circ = \textup{span}\{x_4dx_2\wedge dx_3,x_4dx_1\wedge dx_3, dx_3\wedge dx_5, dx_4 \wedge dx_5\}^\circ,$$
  so, when $x_4\neq 0$ we have $\pp(L)^\circ=\{0\}$, and when $x_4=0$ we have $\pp(L)^\circ=\textup{span}\{ \partial_{x_1},\partial_{x_2}\}$, which coincides with $\ker(\varepsilon)=L\cap TM$.
\end{remark}


\section{Higher presymplectic groupoids}\label{sec:HPSgpd}

Higher Dirac structures have underlying Lie algebroids, and we now identify
the corresponding global objects. As particular cases, these will include the presymplectic groupoids
of \cite{BCWZ} as well as the multisymplectic groupoids of \cite{BCI}, which are the global counterparts of
ordinary Dirac structures and higher Poisson structures, respectively.
Just as these examples, general higher Dirac structures are closely related to
multiplicative differential forms on Lie groupoids, that we briefly recall.

Consider a Lie groupoid $\gpd \rightrightarrows M$; we denote its source and target maps by $\sour, \tar: \gpd\to M$, and
its multiplication map by $\mathsf{m}: \gpd  _{\sour}\times_{\tar}\gpd \to \gpd$;
we often identify $M$ with its image in $\gpd$ by the groupoid identity section.
A differential form $\omega \in \Omega^n(\gpd)$ is called {\em multiplicative} if
$$
\mathsf{m}^*\omega = \pr_1^*\omega + \pr_2^*\omega,
$$
where $\pr_1, \pr_2: \gpd_{(2)}\to \gpd$ are the natural projections.

In order to relate multiplicative forms to higher Dirac structures, we will make use of the infinitesimal description of closed multiplicative forms obtained in \cite{AC,BC1}.
Let $A\to M$ be the Lie algebroid of $\gpd$ (as a vector bundle, $A=\ker(d\sour)|_M \subset T\gpd|_M$, with anchor map $\rho: A\to TM$ induced by $d\tar$ and Lie bracket coming from the identification of sections of $A$ with right-invariant vector fields on $\gpd$). Any closed multiplicative form $\omega \in \Omega^{k+1}(\gpd)$
defines a vector-bundle map
\begin{equation}\label{eq:mu}
  \mu: A\to \wedge^{k}T^*M,\;\;\; \mu(a)(X_1,\ldots,X_{k})=\omega(a,X_1,\ldots,X_{k})
\end{equation}
satisfying the following two properties:
\begin{align}
  & i_{\rho(a)}\mu(b) = - i_{\rho(b)} \mu(a), \label{eq:IM1}\\
  & \mu([a,b])= \Lie_{\rho(a)}\mu(b) - i_{\rho(b)}d\mu(a),\label{eq:IM2}
\end{align}
for all $a,b \in \Gamma(A)$; such map $\mu$ is called a (closed, degree $k+1$) {\em IM-form} on the Lie algebroid $A$. The key relation between $\omega$ and $\mu$ is
\begin{equation}\label{eq:rel}
  i_{a^r}\omega = \tar^* \mu(a),
\end{equation}
for all $a\in \Gamma(A)$, where $a^r$ is the vector field on $\gpd$ obtained by right translation of $a$.
Conversely, it is proven in \cite{AC,BC1} that, if $\gpd$ is source-simply-connected, this correspondence between closed multiplicative forms on $\gpd$ and closed IM-forms on $A$ is a bijection; i.e.,
any closed (degree $n$) IM-form $\mu$ on $A$ as above defines a unique closed, multiplicative form $\omega\in \Omega^{n}(\gpd)$ satisfying \eqref{eq:rel}.

The first step to ``integrate'' a higher Dirac structure $L\subset TM+ \wedge^k T^*M$ is noticing that, once we view $L\to M$ as a Lie algebroid, the projection on the second factor $\pr_2|_L: L \to \wedge^k T^*M$ is a closed IM-form (condition \eqref{eq:IM1} amounts to $L$ being isotropic, while \eqref{eq:IM2} is equivalent to its integrability with respect to the Courant bracket). Note that the fact that $L$ is weakly lagrangian imposes an additional condition on this IM-form, namely, 
$$
\Ima(\pr_2|_L)^\circ  = \ker(\pr_2|_L),
$$
which comes from $L^\perp\cap TM = L\cap TM$. More generally, for a Lie algebroid $A\to M$, any closed IM-form $\mu: A \to \wedge^k T^*M$
such that
\begin{align}
  &\Ima(\mu)^\circ = \rho(\ker(\mu)),\label{eq:extracond}\\
  &\ker(\mu)\cap \ker(\rho) =\{0\}, \label{eq:ker0}
\end{align}
defines a higher Dirac structure $L\subset TM+ \wedge^k T^*M$ as the image of the map
$$
\rho+\mu: A \to TM+ \wedge^k T^*M,
$$
in such a way that $\rho+\mu$ defines a Lie-algebroid isomorphism from $A$ to $L$ (in this sense, these IM-forms are thought of as {\em equivalent}).

\begin{lemma}\label{lem:higherdirac}
  Let $\gpd$ be a Lie groupoid with Lie algebroid $A$, let $\omega \in \Omega^{k+1}(\gpd)$ be a closed multiplicative form and $\mu: A \to \wedge^k T^*M$ be the corresponding IM-form. Then
  \begin{itemize}
  \item[(a)] condition \eqref{eq:ker0} holds if and only if $\ker(\omega)\cap\ker(d\sour)\cap\ker(d\tar)=\{0\}$
    at all points of $\gpd$,
  \item[(b)] condition \eqref{eq:extracond} holds if and only if $d_g\tar(\ker(\omega)\cap \ker(d\sour)) = (\ker(\omega)\cap TM) |_{\tar(g)}$ for all $g\in \gpd$.
  \end{itemize}
\end{lemma}

\begin{proof}
  To verify (a), note that any $X\in \ker(d\sour)|_g$ is the right-translation by $g$ of an element $a \in A|_{\tar(g)}$, so we see from \eqref{eq:rel} that $X \in \ker(\omega)$ if and only if $a\in \ker(\mu)$. Also, $d\tar(X)=d\tar(a^r)=\rho(a)$ (as a consequence of $\tar\circ R_g = \tar$, where $R_g(h)=hg$), so it follows that $X\in \ker(\omega)\cap\ker(d\sour)\cap\ker(d\tar)$ if and only if $X=a^r$ and $a\in \ker(\mu)\cap \ker(\rho)$. This proves (a).

  For (b), it follows from similar arguments that  $d_g\tar(\ker(\omega)\cap \ker(d\sour))=\rho(\ker(\mu))|_{\tar(g)}$. On the other hand, using the fact that the pullback of any multiplicative form by the identity section $M\to \gpd$ vanishes, we see that $\ker(\omega)\cap TM = \Ima(\mu)^\circ$ (cf. \cite[Prop.~1]{BCI}), so (b) follows.
\end{proof}

The previous lemma suggests the following notion of higher presymplectic groupoid:

\begin{definition}
  A {\em $k$-presymplectic groupoid} is a Lie groupoid $\gpd\rightrightarrows M$ equipped with a closed, multiplicative form $\omega\in \Omega^{k+1}(\gpd)$ satisfying, for all $g\in \gpd$,
  \begin{itemize}
  \item[(a)] $(\ker(\w)\cap \ker(d\tar)\cap \ker(d\sour))|_{g}=\{0\}$,
  \item[(b)] $d_g\tar(\ker (\w)\cap \ker(d\sour))=(\ker (\w)\cap TM)|_{\tar(g)}$.   \end{itemize}
\end{definition}

The special case of non-degenerate $\omega$ corresponds to
the {\em multisymplectic groupoids} of \cite{BCI}.
For $k=1$, the previous definition boils down to the presymplectic groupoids of \cite{BCWZ} (in this case,
condition (b) can be replaced by the dimension condition $\mathrm{dim}(\gpd) = 2 \, \mathrm{dim}(M)$, see \cite[Cor.~4.8]{BCWZ}).


\begin{theorem}\ \label{thm:integration of HD}
  \begin{itemize}
  \item[(a)] Let $(\gpd \rightrightarrows M,\w)$ be a $k$-presymplectic groupoid with Lie algebroid $A$. Then
    $M$ inherits a higher Dirac structure $L\subset TM+ \wedge^kT^*M$, naturally isomorphic to $A$ as a Lie algebroid.
  \item[(b)] Let $L \subset TM+ \wedge^k T^*M$ be a higher Dirac structure on $M$
    whose underlying Lie algebroid is integrable, and let $\gpd$ be a source-simply-connected Lie groupoid integrating it. Then $\gpd$ carries a unique closed form $\omega\in \Omega^{k+1}(\gpd)$ making it into a
    $k$-presymplectic groupoid and satisfying, for any $a=X+\alpha \in \Gamma(L)$,
    $$
    i_{a^r}\omega = \tar^*\alpha.
    $$
  \end{itemize}
\end{theorem}

\begin{proof}
  For part (a), note that the IM-form $\mu$ associated with $\omega$ as in \eqref{eq:mu} satisfies \eqref{eq:extracond} and \eqref{eq:ker0} by Lemma~\ref{lem:higherdirac}. So it defines a higher Dirac structure $L$ on $M$ as the image of the map $\rho + \mu: A \to TM+ \wedge^k T^*M$, which is itself a Lie-algebroid isomorphism onto $L$. Part (b) follows from the integration of the closed IM-form $\pr_2|_L: L\to \wedge^k T^*M$ and
  Lemma~\ref{lem:higherdirac}.
\end{proof}

\begin{remark}
  For a higher presymplectic groupoid $(\gpd,\w)$ with induced higher Dirac structure $L = \{\rho(u)+\mu(u), \, u\in A\}$ on $M$, it is a direct verification that condition \eqref{eq:rel} implies that $L$ can be alternatively written as
  $$
  L |_{\tar(g)}=\{ X+\alpha \,|\, X=d\tar(Z), i_Z\omega=\tar^*\alpha \textrm{ for some }\, Z \in T_g\gpd\},
  $$
  for all $g\in \gpd$. This is a natural generalization of the fact that, for $k=1$, the higher Dirac structure $L$ is characterized by $\tar$ being a (forward) Dirac map (see \cite{BCWZ}).
\end{remark}

Considering the natural notions of isomorphism between higher presymplectic group\-oids and higher Dirac structures, one can directly check that the construction in Theorem~\ref{thm:integration of HD} is functorial, and in fact leads to an equivalence of categories, for each $k$, between source-simply-connected higher presymplectic groupoids and  higher Dirac structures whose underlying Lie algebroids are integrable.

The previous theorem recovers the correspondence between multisymplectic group\-oids and higher Poisson structures
of \cite[Sec.~4]{BCI}, where one can find various explicit examples.
At the other extreme, one has higher Dirac structures $L$ satisfying $L\cap \wedge^kT^*M=\{0\}$,  as in Example~\ref{ex:HD-mnfd}.

\begin{example}
When $L$ is a higher Dirac structure such that $L\cap \wedge^kT^*M=\{0\}$,
$E$ is an integrable subbundle of $TM$,
  and $L$ is isomorphic to it as a Lie algebroid. With the identification $L\cong E$, the closed IM-form $\pr_2|_L:L\to \wedge^k T^*M$ is identified with
  the $E$-skew form $\varepsilon: E\to \wedge^k T^*M$ defining $L$.
  It follows that a source-simply-connected Lie groupoid integrating $L$ is identified with the monodromy groupoid $\gpd(E)\rightrightarrows M$ of the foliation tangent to $E$, i.e., the Lie groupoid defined by paths on the leaves of $E$ up to leafwise homotopy \cite{MM} (when $E=TM$, this is just the fundamental groupoid of $M$). The multiplicative $(k+1)$-form $\omega$ on $\gpd(E)$ making into a $k$-presymplectic groupoid can be obtained as follows. Let $\widehat{\varepsilon} \in \Omega^{k+1}(M)$ be an extension of $\varepsilon: E\to \wedge^k T^*M$ (which always exists). We have an induced multiplicative $(k+1)$-form on the pair groupoid $M\times M$ given by $p_1^*\widehat{\varepsilon} -  p_2^*\widehat{\varepsilon}$ (where $p_i: M\times M \to M$ are the natural projections), and a groupoid morphism $(\tar,\sour):\gpd(E)\to M\times M$, defined by source and target maps on $\gpd(E)$ (i.e., initial and end points of paths). One can directly verify that
  $$
  \omega = (\tar,\sour)^* (p_1^*\widehat{\varepsilon} - p_2^*\widehat{\varepsilon}),
  $$
  which only depends on $\varepsilon$, not on the chosen extension.

\end{example}

More generally, one can describe Lie groupoids integrating regular higher Dirac structures, generalizing \cite[Sec.~8.4]{BCWZ}. The discussion is actually valid for general regular, isotropic subbundles.

\begin{example}
Let $L\subset TM+\wedge^k T^*M$ be an integrable, regular, isotropic subbundle, determined by $E$, $A_L$ and $\varepsilon$. As a Lie algebroid, $L$ may be seen as an abelian extension
$$
0\to A_L \to L \to E \to 0,
$$
where $A_L$ carries a representation of $E$ by Lie derivatives. Splitting this sequence is equivalent to picking a lift $\tve\in \Gamma(E^*\otimes \wedge^k\TM)$ of $\ve$, which allows us to identify $L$, as a vector bundle, with the direct sum $(E+ A_L) \to M$: explicitly,
the isomorphism is given by $X+\alpha\mapsto X+\tve(X)+\alpha\in L$. The induced Lie bracket on $\Gamma(E+ A_L)$ is given by
$$
[ X+\alpha,Y+\beta] = [X,Y] + \Lie_X \beta - \Lie_Y \alpha + c(X,Y),
$$
for the cocycle
$$
c(X,Y):= \Lie_X \tve(Y) - i_Y d\tve(X) -\tve([X,Y]) \in\Omega^2(E,A_L).
$$
(The class $[c]\in H^2(E,A_L)$ is independent of the choice of the lift $\tve$ and determines the isomorphism class of the extension.) Hence, as a Lie algebroid, $L$ is isomorphic to a twisted semidirect product $E\ltimes_c A_L$.
It follows that, if the cocycle $c$ integrates to a groupoid cocycle $\widetilde{c}$ on the monodromy groupoid $\gpd(E)$, see \cite[Sec.~8]{BCWZ} and \cite{Cra},  the twisted semi-direct product $\gpd(E)\ltimes_{\widetilde{c}}A_L$ is a Lie groupoid integrating $L$. A description of the $k$-presymplectic form can be adapted from \cite[Cor.~8.6]{BCWZ} for $[c]=0$.
\end{example}

More on the topic of integration of higher Dirac structures can be found in \cite[Sec.~4.3.1]{Martz-th}.

\noindent{\bf Acknowledgments}: This project has had the financial support of CNPq and FAPERJ (H.B.),
CAPES (R.R.) and the {\em Ci\^encias sem Fronteiras} program (N.M.A.), sponsored by CNPq. We thank Juan Carlos Marrero and Marco Zambon for useful comments.

\begin{appendices}

  \section{Alternative notions to lagrangian subspace} \label{sec:App}

  As mentioned in Section~\ref{sec:intro}, Dirac structures on a vector space $V$  are defined as lagrangian subspaces $L\subset V+V^*$, i.e., (C1) $L=L^\perp$, and alternatively defined by
  $$ \textrm{(C2) } L\subseteq L^\perp, \textrm{ and } L\cap V=\pr_2(L)^\circ \textrm{ or } \Ann(L\cap TM)=\pr_2(L).$$
  In \cite{Co} we find a third equivalent way to define such subspaces as
  $$ \textrm{(C3) } L\subseteq L^\perp, \textrm{ and } \pr_1(L)=(L\cap V^*)^\circ \textrm{ or } \Ann(\pr_1(L))=L\cap V^*.$$

  In $V+\wedge^kV^*$, for $k\geq 2$, these notions are not equivalent anymore, and, moreover, $E=\pr_2(L)^\circ$ is not equivalent to $\Ann(E)=\pr_2(L)$. We thus have four alternative ways to extend the notion of lagrangian subspace for $k\geq 2$. We list them using the notation $E=\pr_1(L)$, $A_L=L\cap \wedge^k V^*$:
  \begin{enumerate}
  \item[(C2w)] $L\subseteq L^\perp$ and
    $L\cap V=\pp(L)^\circ$, 
  \item[(C2s)] $L\subseteq L^\perp$ and
    $\Ann(L\cap V)=\pp(L)$. 
  \item[(C3w)] $L\subseteq L^\perp$ and
    $E=A_L^\circ$, 
  \item[(C3s)] $L\subseteq L^\perp$ and $\Ann(E)=A_L$, 
  \end{enumerate}
  We have defined weakly lagrangian as (C2w) in Definition~\ref{def:weakly-lagrangian} and already checked that it is not equivalent to lagrangian in Example~\ref{ex:lHP1}. From (\ref{eq:lag}) and (\ref{eq:wlag}) we have the following result.

  \begin{lemma}\label{prop:general case}
    A lagrangian subspace $L\subset V+\wedge^kV^*$ satisfies {\normalfont (C2w)} and {\normalfont (C3s)}.
  \end{lemma}

   Let us first consider standard isotropic subspaces $L(E,A_L,\varepsilon)$, for which we have $\Ann(E)^\circ=E$, whereas we do not necessarily have $\Ann(\pp(L)^\circ)=\pp(L)$, just $\pp(L)\subseteq \Ann(\pp(L)^\circ)$. Thus, $E=\pp(L)^\circ$ is weaker than $\Ann(E)=\pp(L)$. This, together with the following result, justifies the use of w and s above for weak and strong.


  \begin{proposition}\label{prop:hierarchy with DC}
    For standard isotropic subspaces of $V+\wedge^k V^*$, with $k\geq 2$, we have the hierarchy of different notions
\begin{center}
    \begin{tikzpicture}
      \matrix (m) [matrix of math nodes,row sep=0em,column sep=4em,minimum
      width=2em]
      { & & & \textnormal{(C2w)} \\
        \textnormal{(C2s)} & \textnormal{(C1)}\cong \textnormal{(C3s)} & \textnormal{(C2w)+(C3w)} & \\
        & & & \textnormal{(C3w)} \\};
      \path[-stealth]
      (m-2-1) edge (m-2-2)
      (m-2-2) edge (m-2-3)
      (m-2-3) edge (m-1-4)
      (m-2-3) edge (m-3-4);
    \end{tikzpicture} \end{center}
\noindent where  \textnormal{(C1)} corresponds to lagrangian and \textnormal{(C2w)} to weakly lagrangian subspaces.
  \end{proposition}

  \begin{proof}
  We clearly have the implication (C3s)$\to$(C3w). To see that (C2s)$\to$(C2w), as $L\cap V\subseteq
  E$, we just have to check that for $E=V$, we have $\dim(L\cap V)\leq n-k$ or $\dim (L\cap V)=n$.
  Indeed, for $E=V$, we have $A_L=\{0\}$, so $L=L(V,0,\varepsilon)$ is the graph of a
  $(k+1)$-form $\varepsilon$. When $\varepsilon=0$, we have $L\cap V=V$. When $\varepsilon\neq 0$, as the kernel of a form is of maximal dimension when the form is decomposable,  we get $\dim(L\cap V)=\dim(\ker(\varepsilon))\leq n-(k+1)$, as $\varepsilon$ is a $k+1$-form.

 Note that (C3s) is equivalent to (C1) by Lemma \ref{lem:Lagrangian sbspace}. We also have that (C2s) implies lagrangian. Indeed, for $X+\alpha\in L^\perp$, we have $\alpha\in \Ann(L\cap V)=\pp(L)$, so there is $X'+\alpha\in L$.
    Their difference is $X - X' \in L^\perp\cap V=\pr_2(L)^\circ$, which is $L \cap V$ as (C2s)$\to$(C2w), so
    $X +\alpha \in L$.  The proof is completed by the examples below.
  \end{proof}

  \begin{example}
    The following standard isotropic subspaces show that the notions above are different. We take standard and dual bases $\{e_i\}$ and $\{e^j\}$ of $\RR^n$ and $(\RR^n)^*$, respectively.
    \begin{itemize}
    \item In $\RR^4 + \wedge^2 (\RR^4)^*$, the subspace $\textup{span}\{e_4 + e^1\wedge e^2, e^1\wedge e^3, e^2\wedge e^3 \}$ satisfies (C3w) but not (C2w), hence not (C2s) or (C3s).
    \item In $\RR^6 + \wedge^2 (\RR^6)^*$, the subspace $\textup{span}\{ e_1,e_2+e^3\wedge e^4, e_3 -e^2\wedge e^4, e^5\wedge e^6\}$ satisfies (C2w) but not (C3w), hence not (C2s) or (C3s).
    \item For a proper subspace $S\subset \wedge^k(\RR^n)^*$ such that $S^\circ=\{0\}$ we have that $L=S$ satisfies (C3w)+(C2w) but not (C1). 
    \item In $\RR^5+ \wedge^2 (\RR^5)^*$, the graph of $e^1\wedge e^2 \wedge e^3+ e^1\wedge e^4 \wedge e^5$ satisfies (C3s), but not (C2s). 
    \end{itemize}
  \end{example}

  We finish the study of the standard case by describing (C2s) in more detail.

  \begin{proposition}
    In the standard case, the condition \textnormal{(C2s)} corresponds to the graph of decomposable
    $(k+1)$-forms on $V$ and to the subspaces of the form $E + \Ann(E)$ for $\dim(E)\leq n-k$.
  \end{proposition}

  \begin{proof}
    The map $\varepsilon$ induces an isomorphism $\frac{E}{L\cap V}\cong \frac{\pp(L)}{A_L}$, as $\ker(\varepsilon)=L\cap V$. By looking at the dimensions, setting $\dim(E)=n-a'$, $\dim (L\cap V)=n-a$ and recalling $\pp(L)=\Ann(L\cap V)$, $A_L=\Ann(E)$ (by Proposition \ref{prop:hierarchy with DC}), we have the constraint
    $$a-a'={a \choose k} - {a' \choose k}.$$

    When $a'=0$, i.e., $E=V$, we have that $a={a \choose k}$ is satisfied only for $a=0,k+1$. The case $a=0$ corresponds to the subspace $V$, whereas $a=k+1$ corresponds to $L(V,0,\varepsilon)$ with $\varepsilon$ a $(k+1)$-form with $(k+1)$-dimensional image and $(n-(k+1))$-dimensional kernel. This means that $\varepsilon$ is decomposable. Indeed, take a basis $\{b_1,\ldots, b_{k+1}, c_{k+2},\ldots, c_n\}\subset V$ with $c_j\in \ker(\varepsilon)$, and dual basis $\{b^1,\ldots, b^{k+1}, c^{k+2},\ldots, c^n\}\subset V$. By $\wedge^{k+1} V^* \cong (\wedge^{k+1} V)^*$, we have that $\varepsilon$ is nonzero only in the subspace generated by $b_1\wedge \ldots \wedge b_{k+1}$, so we must have that $\varepsilon$ is a multiple of $b^1\wedge \ldots b^{k+1}$ and hence decomposable.

    On the other hand, for $a'\geq k$ we also have $a\geq k$. Use repeatedly the binomial identity ${b \choose k }={b-1 \choose k-1} + {b-1 \choose k}$ in the constraint to obtain
    $$ a-a'= { a-1\choose k-1} +\ldots+ {a' \choose k-1}.$$
    The RHS has $a-a'$ positive terms. As $k\geq 2$, this constraint is
    not satisfied unless $a=a'$, i.e., $E=L\cap V$ and $\pp(L)=A_L=\Ann(E)$, which
    means that $L=E+\Ann(E)$.
  \end{proof}

For non-standard isotropic subspaces $L(E,A_L,\varepsilon)$ (for which $n-k< \dim(E)<n$) we have the following.

  \begin{proposition}\label{prop:hierarchy without DC}
    A non-standard isotropic subspace $L(E,A_L,\varepsilon)$ always satisfies \textnormal{(C3s)} and never \textnormal{(C3w)} or \textnormal{(C1)}. The properties \textnormal{(C2w)} and \textnormal{(C2s)} are independent from each other.
  \end{proposition}

  \begin{proof}
    From Definition \ref{def:standard}, we have $L=L(E,0,\varepsilon)=gr(\varepsilon)$, so $A_L=L\cap \wedge^kV^*=\{0\}$. Thus, $\Ann(E)=\{0\}=A_L$ and (C3s) is satisfied, whereas $A_L^\circ=V\neq 0$ and (C3w) is not satisfied. Lemma \ref{lem:Lagrangian sbspace} shows that (C1) is not satisfied either. The proof is completed by the examples below.
  \end{proof}

  \begin{example}
    Examples of non-standard isotropic subspaces, using bases as above.
    \begin{itemize}
    \item In $\RR^3 + \wedge^2 (\RR^3)^*$, the subspace $L=\textup{span}\{ e_1,e_2\}$ satisfies (C2s) but not (C2w).

    \item In $\RR^3 + \wedge^2 (\RR^3)^*$, the subspace $L=\textup{span}\{  e_1+e^2\wedge e^3,e_2+e^1\wedge e^3 \}$ satisfies (C2w) but  not (C2s).

    \end{itemize}
  \end{example}

\end{appendices}{}


\bibliographystyle{plain}

\end{document}